\documentclass[12pt]{amsart}
\usepackage[margin=1.15in]{geometry}
\usepackage{amscd, amssymb, amsmath, wasysym}
\usepackage{graphicx}
\usepackage{amsfonts}
\usepackage{mathrsfs}    
\usepackage{eucal}     
\usepackage{latexsym}   
\usepackage{verbatim}   
\usepackage[all]{xy}   
\usepackage[dvipsnames]{xcolor}
\usepackage{bookmark}
\usepackage{subcaption}

\usepackage{microtype}
\usepackage{tikz-cd}

\usepackage{hyperref}
\hypersetup{
	colorlinks=true,
	linkcolor=NavyBlue,
	filecolor=NavyBlue,
	citecolor = TealBlue,
	urlcolor=magenta,
}

\newcounter{thmcounter}

\numberwithin{equation}{section}
\numberwithin{thmcounter}{section}

\theoremstyle{theorem}
\newtheorem{theoremalpha}{Theorem}
\newtheorem{propositionalpha}[theoremalpha]{Proposition}

\newtheorem*{EKS}{Eisenbud--Koh--Stillman Conjecture}
\newtheorem*{Sidman-Smith}{Sidman--Smith Conjecture}
\newtheorem*{HLMP}{Han--Lee--Moon--Park Conjecture}

\newtheorem{theorem}[thmcounter]{Theorem}
\newtheorem{proposition}[thmcounter]{Proposition}
\newtheorem{lemma}[thmcounter]{Lemma}
\newtheorem{corollary}[thmcounter]{Corollary}

\theoremstyle{definition}

\newtheorem{remark}[thmcounter]{Remark}

\newtheoremstyle{claim}{9pt}{3pt}{}{\parindent}{\bf}{.}{1em}{}

\theoremstyle{claim}

\newcommand{\thistheoremname}{}
\newtheorem*{genericthm*}{\thistheoremname}
\newenvironment{namedthm*}[1]
{\renewcommand{\thistheoremname}{#1}%
	\begin{genericthm*}}
	{\end{genericthm*}}
\newenvironment{namedtheorem*}[1]
{\renewcommand{\thistheoremname}{#1}%
	\begin{genericthm*}}
	{\end{genericthm*}}

\newenvironment{namelist}[1]{%
	\begin{list}{}
		{
			\settowidth{\labelwidth}{#1}%
			\setlength{\labelsep}{0.3em}%
			\setlength{\leftmargin}{\labelwidth}%
			\addtolength{\leftmargin}{\labelsep}}}{%
\end{list}}

\newcommand{\nZ}{\mathbf{Z}}                   
               
\newcommand{\nC}{\mathbf{C}}                     
           
\newcommand{\kk}{\mathbf{k}}         

\newcommand{\nP}{\mathbf{P}}                     
\newcommand{\nA}{\mathbf{A}}

\newcommand{\sO}{\mathscr{O}}

\DeclareMathOperator{\id}{id}

\DeclareMathOperator{\pr}{pr}

\DeclareMathOperator{\res}{res}

\DeclareMathOperator{\Spec}{Spec}

\DeclareMathOperator{\reg}{reg}

\newcommand*{\longhookrightarrow}{\ensuremath{\lhook\joinrel\relbar\joinrel\rightarrow}}

\newcommand*{\longtwoheadrightarrow}{\ensuremath{\joinrel\relbar\joinrel\twoheadrightarrow}}

\newcounter{rkcounter}             
\begin{document}

\title[Determinantal ideals of secant varieties]{Determinantal ideals of secant varieties}

\author{Daniele Agostini}
\address{Fachbereich Mathematik, University of T\"ubingen, Auf der Morgestelle 10, T\"ubingen 72072, Germany}
\email{daniele.agostini@uni-tuebingen.de}

\author{Jinhyung Park}
\address{Department of Mathematical Sciences, KAIST, 291 Daehak-ro, Yuseong-gu, Daejeon 34141, Republic of Korea}
\email{parkjh13@kaist.ac.kr}

\date{\today}

\begin{abstract} 
Using Hilbert schemes of points, we establish a number of results for a smooth projective variety $X$ in a sufficiently ample embedding. If $X$ is a curve or a surface, we show that the ideals of higher secant varieties are determinantally presented, and we prove the same for the first secant variety if $X$ has arbitrary dimension. This completely settles a conjecture of Eisenbud--Koh--Stillman for curves and partially resolves a conjecture of Sidman--Smith in higher dimensions. If $X$ is a curve or a surface we also prove that the corresponding embedding of the Hilbert scheme of points $X^{[d]}$ into the Grassmannian is projectively normal. Finally, if $X$ is an arbitrary projective scheme in a sufficiently ample embedding, then we demonstrate that its homogeneous ideal is generated by quadrics of rank three, confirming a conjecture of Han--Lee--Moon--Park. Along the way, we check that the Hilbert scheme of three points on a smooth variety is the blow-up of the symmetric product along the big diagonal. 
\end{abstract}

\maketitle


\section*{Conventions}
\noindent Throughout the paper, we work over an algebraically closed field $\kk$ of characteristic zero. When $V$ is a $\kk$-vector space, $\nP(V)$ is the space of hyperplanes in $V$. For a projective scheme $X \subseteq \nP (H^0(X, L))$ embedded by the complete linear system of a very ample line bundle $L$ on $X$, we denote by $I(X)=I(X,L)$ the homogeneous ideal defining $X$ in $\nP(H^0(X,L))$. We say that a  property $\mathscr{P}$ holds for a \emph{sufficiently ample line bundle} on a fixed projective scheme $X$ over $\kk$ if there is a line bundle $L_0$ on $X$ such that $\mathscr{P}$ holds for any line bundle $L = L_0\otimes A$ with $A$ ample.

\section{Introduction}

\noindent Let $X$ be a smooth projective variety with an embedding $\varphi\colon X \hookrightarrow \mathbf{P}^r := \mathbf{P}(H^0(X,L))$ given by the complete linear system of a very ample line bundle $L$ on $X$. For an integer $k\geq 0$, the $k$-th \emph{secant variety} $\Sigma_k(X)=\Sigma_k(X,L)$ is the closure of the union of all $k$-planes spanned by $k+1$ distinct points on $X$. We have natural inclusions
\[ 
X = \Sigma_0(X) \subseteq \Sigma_1(X) \subseteq \Sigma_2(X) \subseteq \cdots.
\]
An important topic in the study of secant varieties is the search of an explicit set of equations for them. It is well known that $\Sigma_k(X)$ is not contained in hypersurfaces of degree $k+1$, i.e., $I(\Sigma_k(X,L))_{k+1}=0$  (see e.g., \cite{CJ01}). On the other hand, there is a natural set of determinantal equations of degree $k+2$ constructed as follows. Assume that $L = A\otimes B$ for two line bundles $A,B$ on $X$. Then there is a linear map
\begin{equation}\label{eq:mABk+2} 
	m_{A,B}^{k+2}\colon  \wedge^{k+2}H^0(X,A)\otimes \wedge^{k+2}H^0(X,B) \longrightarrow  I(\Sigma_k(X,L))_{k+2}
\end{equation}
given by 
\[
(s_1 \wedge \cdots \wedge s_{k+2}) \otimes (t_1 \wedge \cdots \wedge t_{k+2}) \longmapsto  \sum_{\sigma \in \mathfrak{S}_{k+2}} \varepsilon_{k+2}(\sigma) m_{A,B}^1(s_{\sigma(1)} \otimes t_1) \cdots m_{A,B}^1(s_{\sigma(k+2)} \otimes t_{k+2}).
\]
where
$$
m_{A,B}^1 \colon H^0(X, A) \otimes H^0(X, B) \longrightarrow H^0(X, L)
$$
is the usual multiplication map on $X$ and 
$$
\varepsilon_{k+2} \colon \mathfrak{S}_{k+2} \longrightarrow \{\pm 1\}
$$
is the alternating character of the symmetric group $\mathfrak{S}_{k+1}$. The image of the map $m^{k+2}_{A,B}$ in \eqref{eq:mABk+2} is generated by the $(k+2) \times (k+2)$-minors of the  \emph{catalecticant matrix} 
\begin{equation}\label{eq:omega}
	\operatorname{Cat}(A,B):=\begin{pmatrix}
		m_{A,B}^1(\alpha_1 \otimes \beta_1) & m_{A,B}^1(\alpha_1 \otimes \beta_2) & \cdots & m_{A,B}^1(\alpha_1 \otimes \beta_b) \\
		m_{A,B}^1(\alpha_2 \otimes \beta_1) & m_{A,B}^1(\alpha_2 \otimes \beta_2)  & \cdots & m_{A,B}^1(\alpha_2 \otimes \beta_b) \\
		\vdots & \vdots & \ddots & \vdots \\
		m_{A,B}^1(\alpha_a \otimes \beta_1)  & m_{A,B}^1(\alpha_a \otimes \beta_2)  & \cdots & m_{A,B}^1(\alpha_a \otimes \beta_b) 
	\end{pmatrix}
\end{equation}
which is an $a \times b$ matrix of linear forms, where $\alpha_1, \ldots, \alpha_a$ is a basis of $H^0(X, A)$, and $\beta_1, \ldots, \beta_b$ is a basis of $H^0(X, B)$. 
This matrix is $1$-generic, and it is straightforward to see that $\operatorname{Cat}(A,B)$ has rank at most $1$ on $X$ so that it has rank at most $k+1$ on $\Sigma_k(X)$ (see \cite[Proposition 6.10]{Eis2005}). Thus its $(k+2)\times (k+2)$-minors are indeed in the homogeneous ideal $I(\Sigma_{k}(X,L))$. Notice that if $I(\Sigma_{k}(X,L))$ is generated in degree $k+2$, then $I(\Sigma_{k}(X,L))$ is generated by $(k+2)\times (k+2)$-minors of $\operatorname{Cat}(A,B)$ if and only if $m_{A,B}^{k+2}$ is surjective. If there is a factorization $L=A\otimes B$ such that $I(\Sigma_k(X,L))$ is generated by the $(k+2) \times (k+2)$-minors of $\operatorname{Cat}(A,B)$, we say that the ideal $I(\Sigma_k(X,L))$ is \emph{determinantally presented}. When $X=C$ is a smooth curve, it was conjectured by Eisenbud--Koh--Stillman \cite[Remark in page 518]{EKS1988} that this is the case if the degree of $L$ is large enough with respect to the genus $g=g(C)$ and the order $k$ of the secant variety. 

\begin{EKS}
Let $C$ be a smooth projective curve of genus $g$. If $A,B$ are two line bundles of sufficiently large degree with respect to $g$ and $k$ on $C$ and $L:=A \otimes B$, then the ideal $I(\Sigma_k(C,L))$ is generated by the $(k+2) \times (k+2)$-minors of $\operatorname{Cat}(A,B)$. In particular, $I(\Sigma_k(C,L))$ is determinantally presented if $\deg L$ is large enough.
\end{EKS}

This was proven for the case $k=0$ by Eisenbud--Koh--Stillman \cite[Theorem 1]{EKS1988}, where the $0$-th secant variety is the curve itself, i.e., $\Sigma_0(C,L)=C$. In the case of an arbitrary secant variety $\Sigma_k(C)$, this was known for the case $g=0$ by Wakerling (see \cite[Remark in page 518]{EKS1988}) and the case $g=1$ by Fisher \cite[Theorem 1.3]{Fisher}. In general, this was only verified set-theoretically by Ravi \cite{Ravi} and scheme-theoretically by Ginensky \cite[Theorem 7.4]{Gin2010}. On the other hand, it was shown by Ein--Niu--Park \cite[Theorem 1.2]{ENP} that the ideal $I(\Sigma_k(C,L))$ is generated in degree $k+2$ as soon as $\deg L \geq 2g+2k+2$. Our first main result is a complete solution to the strongest ideal-theoretic form of the Eisenbud--Koh--Stillman conjecture.\footnote{While we were preparing this paper, we learned that Junho Choe also had independently obtained a similar result (see \cite[Theorem 1.4]{Choe25+}).}

\begin{theoremalpha}\label{thm:EKS}
Let $C$ be a smooth projective curve of genus $g$, and $A,B$ be two line bundles on $C$ with $\deg(A),\deg(B)\geq 2g+k+1$ such that $A \not\cong B$ when $g > 0$ and $\deg (A)=\deg(B)=2g+k+1$. Then the ideal $I(\Sigma_k(C,A\otimes B))$ is generated by the $(k+2) \times (k+2)$-minors of $\operatorname{Cat}(A,B)$. Furthermore, if $L$ is a line bundle on $C$ with
\[
\deg(L) \geq  \min\{ 4g+2k+2,3g+3k+4\}, 
\]
then the ideal $I(\Sigma_k(C,L))$ is determinantally presented.
\end{theoremalpha}

We can actually give more precise conditions on $A,B$ to satisfy the theorem (see Theorem \ref{thm:precisethmA} and Corollary \ref{cor:precisethmA}). Our  bounds are comparable with those of \cite{EKS1988,Fisher, Gin2010, Ravi}. Especially, \cite[Theorem 1]{EKS1988} says that $I(C,L)$ is determinantally presented when $\deg L \geq 4g+2$ (this is the case $k=0$ in our theorem). Our result shows that $\deg L \geq 3g+4$ is sufficient in this case. This already significantly improves the previous result as soon as $g \geq 3$. 
Our theorem would be very useful in finding explicit equations for algebraic curves and their secant varieties.

\medskip

Theorem \ref{thm:EKS} can be generalized to secant varieties of smooth projective surfaces and to the first secant varieties of arbitrary smooth projective varieties without the explicit effective conditions on $A,B$. This is our second main result:

\begin{theoremalpha}\label{thm:SS}
Let $X$ be a smooth projective variety of dimension $n$. Assume that $n\leq 2$ or $k\leq 1$. If $L=A\otimes B$ is a line bundle on $X$ with $A,B$ sufficiently ample, then the ideal $I(\Sigma_k(X,L))$ is generated by  the $(k+2) \times (k+2)$-minors of $\operatorname{Cat}(A,B)$. In particular, the ideal $I(\Sigma_k(X,L))$ is determinantally presented whenever $L$ is sufficiently ample.
\end{theoremalpha}

This theorem gives an answer to Eisenbud's question \cite[Question 1.5]{BGL13} in the cases $n \leq 2$ or $k \leq 1$. This result was proven by Sidman--Smith \cite{SS2011} for an arbitrary projective scheme $X$ in the case $k=0$ and by Raicu \cite{Raicu} in the case $k=1$ for $X=\nP^n$. 

\medskip

For higher secant varieties, one important issue that arises is that the $(k+2) \times (k+2)$-minors of $\operatorname{Cat}(A,B)$ do not vanish only on $\Sigma_k(X,L)$ but also on the $k$-th \emph{cactus variety} $\kappa_k(X,L)$, which is the closure of the union of all $k$-planes spanned by finite subschemes $\xi\subseteq X$ of length $k+1$. It is always true that $\Sigma_k(X,L)\subseteq \kappa_k(X,L)$ and equality holds when the \emph{Hilbert scheme}
\[ 
X^{[k+1]} := \{ \xi \subseteq X \text{ finite subscheme } \,|\, \operatorname{length}(\xi) = k+1 \} 
\]
of $k+1$ points on $X$ is irreducible. In particular, this holds if $n\leq 2$ or $k\leq 2$: in these cases the Hilbert scheme of points is even smooth. In general, however, this is not the case if $n$ and $k$ are large enough, so that Sidman--Smith proposed to generalize the Eisenbud--Koh--Stillman conjecture by replacing secant varieties with cactus varieties (see \cite[Conjecture 1.2 and Subsequent Discussion]{SS2011}).

\begin{Sidman-Smith}
Let $X$ be a smooth projective variety of dimension $n$. If $L=A\otimes B$ is a line bundle on $X$ with $A,B$ sufficiently ample, then the ideal $I(\kappa_k(X,L))$ is generated by  the $(k+2) \times (k+2)$-minors of $\operatorname{Cat}(A,B)$. In particular, the ideal $I(\kappa_k(X,L))$ is determinantally presented whenever $L$ is sufficiently ample.
\end{Sidman-Smith}

Theorem \ref{thm:SS} is a full affirmative answer to this conjecture in the cases $n\leq 2$ or $k\leq 1$. As remarked before, in these cases, $\Sigma_k(X,L)=\kappa_k(X,L)$. We point out that if $n\leq 2$ or $k\leq 2$, then it was recently proven by Choi--Lacini--Park--Sheridan \cite[Theorem B]{CLPS25+} that the ideal $I(\Sigma_k(X,L)) = I(\kappa_k(X,L))$ is generated in degree $k+2$ whenever $L$ is sufficiently ample. This is indeed the starting point for this paper. 

\medskip

The distinction between cactus varieties and secant varieties becomes problematic even when $X$ is a singular curve. It is known that the Sidman--Smith conjecture does not hold for secant varieties of singular curves (see \cite[Theorem 1.17]{BGL13}). However, for an arbitrary cactus variety of any projective scheme, the conjecture was recently proven by Buczy\'nska--Buczy\'nski--Farnik \cite{BBF24+} but at the level of sets, not of ideals. Actually, they expect that the conjecture should not always hold on the ideal-theoretic or even scheme-theoretic level (see \cite[Subsection 6.3]{BBF24+}).

\medskip

One of our approaches to Theorem \ref{thm:SS} also leads to a quick proof of projective normality for Hilbert schemes of points on curves and surfaces. Suppose that $X$ is smooth of dimension $n\leq 2$, set $k\geq 0$ and consider a line bundle $L$ on $X$ that is $(k+1)$-very ample, meaning that for any $\xi\in X^{[k+2]}$ the evaluation map
\[ \operatorname{ev}_{L,\xi}\colon H^0(X,L) \longrightarrow H^0(X,L\otimes \sO_{\xi}) \]
is surjective. This yields a map 
\[ \varphi_{L,k+2}\colon X^{[k+2]} \longrightarrow G(k+2,H^0(X,L)) \subseteq \nP(\wedge^{k+2} H^0(X,L)); \quad \xi \longmapsto H^0(X,L\otimes \sO_{\xi}) \]
into the Grassmannian of $(k+2)$-dimensional quotients of $H^0(X,L)$. This map is furthermore an embedding if $L$ is $(k+2)$-very ample by \cite{CG}, so that it realizes $X^{[k+2]}$ as a subvariety of $\nP(\wedge^{k+2}H^0(X,L))$. It is quite natural then to expect that the positivity $L$ on $X$ should be reflected into the equations of $X^{[k+2]}$. Here we confirm this expectation for projective normality:

\begin{theoremalpha}\label{thm:projnormalityhilb}
	Let $X$ be a smooth projective variety of dimension $n\leq 2$, and let $k\geq 0$. If $L$ is a sufficiently ample line bundle on $X$, then the embedding
	\[ \varphi_{L,k+2}\colon X^{[k+2]} \longhookrightarrow \nP(\wedge^{k+2}H^0(X,L)) \]
	is projectively normal.
\end{theoremalpha}

We expect that a similar result should hold in general for higher syzygies of $X^{[k+2]}$, in terms property $(N_p)$ of Green and Lazarsfeld. This type of questions was first studied by Sheridan \cite{Sheridan} for the case $n=1$ and $k=0$ (see also \cite{SeoYoon} for the cases $n \geq 2$ and $k=0$). Another very interesting case to investigate in detail would be the Hilbert scheme of points on a $K3$ surface, since it is a canonical example of a hyperk\"ahler variety, whose syzygies are still largely unexplored.
\medskip

Perhaps surprisingly, it turns out that our proofs of Theorem \ref{thm:SS} and Theorem \ref{thm:projnormalityhilb} go through commutative algebra. Assume $n\leq 2$ or $k\leq 1$, and let $\kk[x_{ij}]$ be the coordinate ring of $(\nA^n)^{k+2} = \nA^{n(k+2)}$, seen as the set of matrices 
\begin{equation}\label{eq:matrixM}  
\begin{pmatrix}
	x_{11} & x_{12} & x_{13} & \cdots & x_{1,k+2} \\
	x_{21} & x_{22} & x_{23} & \cdots & x_{2,k+2} \\
	\vdots & \vdots & \vdots & \ddots & \vdots \\
	x_{n1} & x_{n2} & x_{n3} & \cdots & x_{n,k+2}	 
\end{pmatrix} 
\end{equation}
where each column represents a point in $\nA^n$. The big diagonal $\Delta_{k+2}$ of the $(k+2)$-th symmetric product $ (\nA^n)^{(k+2)}$ of $\nA^n$ is the set where at least two distinct columns  of \eqref{eq:matrixM} coincide and $I(\Delta_{k+2})\subseteq \kk[x_{ij}]$ is the corresponding ideal. The action of the symmetric group $\mathfrak{S}_{k+2}$ on the columns of \eqref{eq:matrixM} defines the ring $\kk[x_{ij}]^{\mathfrak{S}_{k+2}}$ of multisymmetric polynomials, as well as the set $\kk[x_{ij}]^{\varepsilon_{k+2}}$ of alternating polynomials and the ideal $J_{k+2} := (\kk[x_{ij}]^{\varepsilon_{k+2}}) \subseteq \kk[x_{ij}]$ generated by  the alternating polynomials inside the coordinate ring of $(\nA^n)^{k+2}$. Since an alternating polynomial vanishes if two distinct columns of \eqref{eq:matrixM} are the same, we see that $J_{k+2}\subseteq I(\Delta_{k+2})$. Theorem \ref{thm:SS} is actually implied by a partial converse of this statement, which has also a geometric interpretation in terms of the Hilbert scheme of points. To state it, recall that there is a Hilbert--Chow morphism
\[ 
h_{k+2} \colon X^{[k+2]} \longrightarrow X^{(k+2)}
\]
from the Hilbert scheme of $k+2$ points on $X$ to the $(k+2)$-th symmetric product of $X$, that sends a finite scheme to the points on its support counted with multiplicity. Furthermore, we denote by $\Delta_{(k+2)} \subseteq X^{(k+2)}$ the big diagonal, which is the image of the set $E_{k+2}\subseteq X^{[k+2]}$ of nonreduced subschemes.

\begin{propositionalpha}\label{prop:HC=blow-up}
Let $X$ be a smooth projective variety of dimension $n$. 
Assume that $n\leq 2 $ or $k\leq 1$ as before. Then one has the following:
	\begin{enumerate}
		\item $J_{k+2}^{\mathfrak{S}_{k+2}} = I(\Delta_{k+2})^{\mathfrak{S}_{k+2}}$.
		\item The Hilbert--Chow morphism  $h_{k+2} \colon X^{[k+2]} \to X^{(k+2)}$ is the blow-up of $X^{(k+2)}$ along the big diagonal $\Delta_{(k+2)}$ taken with its reduced structure.
	\end{enumerate}
\end{propositionalpha}

Notice that this statement covers all cases where the Hilbert scheme $X^{[k+2]}$ is smooth. Actually, the only new results here are for the case $k=1$ and $n$ arbitrary.  In fact, in all other cases the stronger statement $J_{k+2} = I(\Delta_{k+2})$ holds. This is easy to see directly for either $n=1$ and arbitrary $k$ or arbitrary $n$ and $k=0$. If instead  $n=2$ and $k$ is arbitrary, this is a highly nontrivial result due to Haiman \cite[Corollary 3.8.3]{Haiman} (see also \cite[Theorem 1.1]{Haimanlectures}). This in turn implies the description of the Hilbert--Chow morphism as a blow-up (Proposition \ref{prop:idealXmsmooth}).

\medskip

Finally, we turn to a more detailed investigation of the homogeneous ideal $I(X,L)$ defined by an arbitrary projective scheme $X$ and a sufficiently ample line bundle $L$ on $X$. The result of Sidman--Smith \cite{SS2011} shows that $I(X,L)$  is generated by the $2\times 2$-minors of a matrix of linear forms.  In particular, $I(X,L)$ is generated by quadrics of rank $4$. Note that a quadric of rank 2 is the product of two linear forms. Thus the minimal rank of a quadric containing $X$ is at least 3. Indeed, it was shown by Han--Lee--Moon--Park \cite[Theorem 1.3]{HLMP2021} that $I(X, L^{\otimes 2})$ is generated by quadrics of rank $3$, and it was asked in \cite[Conjecture 6.1]{HLMP2021} whether the same holds for the ideal $I(X,L)$.
 
\begin{HLMP}
If $X$ is a projective scheme and $L$ is a sufficiently ample line bundle on $X$, then the ideal $I(X,L)$ is generated by quadrics of rank $3$.
\end{HLMP}

Our last main result confirms this conjecture, and also gives an effective condition on $L$ when $X$ is smooth. 

\begin{theoremalpha}\label{thm:HLMP}
If $X$ is a projective scheme and $L$ is a sufficiently ample line bundle on $X$, then the ideal $I(X,L)$ is generated by quadrics of rank $3$. Furthermore, if  $X$ is a smooth projective variety of dimension $n$ and $L:=\omega_X^{\otimes 3}\otimes H^{\otimes (3n+6)} \otimes M$ with $H$ very ample and $M$ nef, then the ideal $I(X,L)$ is generated by quadrics of rank $3$. \end{theoremalpha}

It is worth noting that the first statement of the theorem holds in positive characteristic different from 2 (see Remark \ref{rem:positivechar}). This settles the Han--Lee--Moon--Park conjecture in full generality.

\subsection{Strategy of the proofs}
We first sketch the strategy behind our proofs of Theorem \ref{thm:EKS} and Theorem \ref{thm:SS}. We assume that $X$ is a smooth projective variety of dimension $n$. Fix $k$ such that $n\leq 2$ or $k\leq 1$. As mentioned before, our starting point are the results of Ein--Niu--Park \cite[Theorem 1.2]{ENP} and Choi--Lacini--Park--Sheridan \cite[Theorem B]{CLPS25+}, which imply that the ideal $I(\Sigma_k(X,L))$ is generated in degree $k+2$ if $L$ is sufficiently ample. This shows that the map $m^{k+2}_{A,B}$ in \eqref{eq:mABk+2}  is surjective if and only if the $(k+2)\times (k+2)$ minors of the $1$-generic matrix $\operatorname{Cat}(A,B)$ in \eqref{eq:omega} generate the ideal $I(\Sigma_k(X,L))$. The surjectivity of $m^{k+2}_{A,B}$ is the main focus of this paper.

\medskip

The key idea, inspired by the work of Ein--Lazarsfeld \cite{EL15} on syzygies of algebraic curves, is to interpret the map $m^{k+2}_{A,B}$ as a multiplication map on the Hilbert scheme $X^{[k+2]}$.  Recall that if $A$ is any line bundle on $X$, there are a corresponding tautological bundle $E_{k+2,A}$ and its determinant line bundle $N_{k+2,A}:=\det E_{k+2,A}$ on $X^{[k+2]}$. It turns out that if $A,B$ are sufficiently ample, the map $m^{k+2}_{A,B}$ coincides with the multiplication map of global sections
\begin{equation}\label{eq:multmaphilb}
m^{k+2}_{A,B}\colon H^0(X^{[k+2]},N_{k+2,A})\otimes H^0(X^{[k+2]},N_{k+2,B}) \longrightarrow H^0(X^{[k+2]},N_{k+2,A}\otimes N_{k+2,B}).
\end{equation}

\medskip

In this paper, we give two different proofs of Theorem \ref{thm:SS}. The first proof goes through Proposition \ref{prop:HC=blow-up}. The multiplication map $m^{k+2}_{A,B}$ on $X^{[k+2]}$ can also be interpreted as a multiplication map of global sections
\begin{equation*}
	m^{k+2}_{A,B}\colon H^0(X^{(k+2)},h_*N_{k+2,A})\otimes H^0(X^{(k+2)},h_*N_{k+2,B}) \longrightarrow H^0(X^{(k+2)},h_*(N_{k+2,A}\otimes N_{k+2,B})) 
\end{equation*} 
on the symmetric product $X^{(k+2)}$. 
It holds that 
$$
h_*N_{k+2,A} \cong h_*N_{k+2,\sO_X}\otimes S_{k+2,A}~~\text{ and }~~h_*N_{k+2,B} \cong h_*N_{k+2,\sO_X}\otimes S_{k+2,B},
$$ 
where $h \colon  X^{[k+2]} \rightarrow X^{(k+2)}$ is the Hilbert--Chow morphism.
Note that $S_{k+2,A},S_{k+2,B}$ are sufficiently ample line bundles on $X^{(k+2)}$ if $A,B$ are on $X$. We show that the map $m^{k+2}_{A,B}$ is surjective for $A,B$ sufficiently ample if and only if the map of sheaves 
\[
h_*(N_{k+2,\sO_{X}}) \otimes h_*(N_{k+2,\sO_{X}}) \longrightarrow h_*(N_{k+2,\sO_X}\otimes N_{k+2,\sO_X})
\]
is surjective as well. This is a local statement on the symmetric product $X^{(k+2)}$. Taking local coordinates, we can reduce the problem to the case that $X=\nA^{n}$. In this case, the surjectivity of the map of sheaves is equivalent to the equality $J_{k+2}^{\mathfrak{S}_{k+2}} = I(\Delta_{k+2})^{\mathfrak{S}_{k+2}}$ asserted in Proposition \ref{prop:HC=blow-up}. In a similar way, we show that the statement of Theorem \ref{thm:projnormalityhilb} is implied by the surjectivity of the maps
\[ h_*(N_{k+2,\sO_X})^{\otimes \ell} \longrightarrow h_*(N_{k+2,\sO_X}^{\otimes \ell})  \]
for all $\ell \geq 1$. We reduce again to the case $X=\nA^n$ where the surjectivity is in turn a consequence of the stronger statement $J_{k+2}=I(\Delta_{k+2})$, which holds for $n\leq 2$: see Remark \ref{rmk:equalityideals}.

\medskip

The second proof of Theorem  \ref{thm:SS} is via cohomology vanishing on Hilbert schemes of points. 
It might be more involved than the previous algebraic approach, but it points to the statement that one would like to prove in order to obtain an effective result. In fact, Theorem \ref{thm:EKS} is shown by this way. We work on the Hilbert scheme $X^{[k+2]}$ and, if $B$ is sufficiently ample, then $E_{k+2,B}$ and $N_{k+2,B}$ are globally generated so that standard arguments yield an exact sequence
\begin{equation}
	\cdots  \longrightarrow G_2 \longrightarrow G_1  \longrightarrow  H^0(X^{[k+2]},N_{k+2,B}) \otimes N_{k+2,A}  \longrightarrow  N_{k+2,A}\otimes N_{k+2,B}  \longrightarrow  0,
\end{equation}
where each $G_i$ is a direct sum of copies of $ S^iE_{k+2,B}^{\vee} \otimes N_{k+2,A}$. Hence, to prove the surjectivity of the multiplication map $m_{A,B}^{k+2}$, it is enough to show that if $A$ is sufficiently ample, then
\begin{equation}\label{eq:h1vanishing} 
H^i(X^{[k+2]},S^iE_{k+2,B}^{\vee}\otimes N_{k+2,A}) = 0 \qquad \text{ for all } i>0.
\end{equation}
For this purpose, we use a strategy already employed by the authors in \cite{Ago22,CLPS25+} via the Hilbert--Chow morphism
$h \colon  X^{[k+2]} \rightarrow X^{(k+2)}$. Recall  that $N_{k+2,A} \cong N_{k+2,\sO_X}\otimes h^*S_{k+2,A}$ for a sufficiently ample line bundle $S_{k+2,A}$ on $X^{(k+2)}$. Thus Fujita--Serre vanishing (see e.g., \cite[Theorem 1.4.35]{LazPosI}) and the Leray spectral sequence show that the vanishing \eqref{eq:h1vanishing} for $A$ sufficiently ample is equivalent to
\begin{equation}\label{eq:r1vanishing} 
	R^ih_*(S^iE_{k+2,B}^{\vee}\otimes N_{k+2,\sO_X}) = 0 \qquad \text{ for all } i>0.
\end{equation}
This is immediate if  $n=1$ because then the Hilbert--Chow morphism is an isomorphism. If instead $n=2$, this was proven in \cite[Proposition 4.9]{CLPS25+}, and if $k=0$, we prove it in Proposition \ref{prop:cohvanishing}. When $k=1$, we do not know whether the vanishings of \eqref{eq:r1vanishing} hold for an arbitrary $n$, even though we believe that they do. However, using the nested Hilbert schemes, we can reduce in Lemma \ref{lem:cohvanishingbetter} the surjectivity of $m_{A,B}^{k+2}$ to a vanishing like \eqref{eq:r1vanishing} but on $X^{[k+1]}$ instead that on $X^{[k+2]}$. This vanishing holds for $k=1$ thanks to Proposition \ref{prop:cohvanishing}. Now, for the effective statement of Theorem \ref{thm:EKS}, it is not enough to just invoke Fujita--Serre vanishing since we need a quantitative statement. We obtain Theorem \ref{thm:EKS} via the yoga of tautological  bundles on $C^{[k+2]}$, which has recently proven itself very useful for various problems on syzygies of algebraic curves and their secant varieties (see \cite{Ago24, CKP23+, EL15, ENP, NP24+}).

\medskip

Finally, let us discuss our proof of Theorem \ref{thm:HLMP} so that $X$ is an arbitrary projective scheme. The key point is again given by the surjectivity of two maps
\begin{align*}
m^2_{A,B} &\colon \wedge^2H^0(X,A)\otimes \wedge^2 H^0(X,A) \to I(X, L)_2	\\  
s^2_{A,B} & \colon S^2H^0(X,A)\otimes \wedge^2 H^0(X,A) \to \wedge^2 H^0(X,L)  
\end{align*}
when $A,B$ are sufficiently ample line bundles on $X$ and $L:=A \otimes B$. 
If $X$ is smooth, then the surjectivity of these maps can be shown via cohomology vanishing on Hilbert schemes of points as before. This was the approach we originally employed. However, we realized that the surjectivity can be established without relying on Hilbert schemes although the arguments are certainly inspired by the proof of Theorem \ref{thm:SS}. This latter approach is only written in this paper, but we expect that the reader could easily reproduce our original proof via  Hilbert schemes of points when $X$ is smooth. 
In the process, we reprove the main results of \cite{SS2011} that the ideal $I(X, L)$ is determinantally presented with an effective condition on $X$ when $X$ is smooth. Another important ingredient of the proof of Theorem \ref{thm:HLMP} is a special case of  \cite[Theorem 1.1]{HLMP2021} that the ideal  $I(\nP^n,\sO_{\nP^n}(2))$ of the second Veronese embedding is generated by quadrics of rank 3. Actually, we give a quick proof of this result via representation theory (Proposition \ref{prop:rank3veronese}). Combining this with Theorem \ref{thm:HLMP}, we fully recover \cite[Theorem 1.1]{HLMP2021}, which is the main result of that paper. Our proof of Theorem \ref{thm:HLMP} is fairly self-contained.

\subsection{Organization of the paper}
We begin in Section \ref{sec:background} with presenting miscellaneous technical statements including a special case of Proposition \ref{prop:HC=blow-up}. In Section \ref{sec:yoga}, we recall some facts on Hilbert schemes of points and show key cohomology vanishings that we are going to use later. We also complete the proof of Proposition \ref{prop:HC=blow-up}. Section \ref{sec:algebra} is devoted to the first proof of Theorem \ref{thm:SS}, and Section \ref{sec:projnorm} takes up similar ideas to obtain Theorem \ref{thm:projnormalityhilb}. In Section \ref{sec:cohomology}, we instead employ a cohomological approach to prove Theorem \ref{thm:EKS} and to give the second proof of Theorem \ref{thm:SS}. Finally, in Section \ref{sec:quadrics43}, we show Theorem \ref{thm:HLMP} and provide a quick alternative approach to the main results of \cite{HLMP2021} and \cite{SS2011} along the way.

\subsection*{Acknowledgments} We are grateful to Junho Choe, Doyoung Choi, Joachim Jelisiejew, Hanieh Keneshlou, Euisung Park, Claudiu Raicu, and John Sheridan for useful comments and discussions. DA thanks the Department of Mathematical Sciences at KAIST for the hospitality during the KAIST Thematic Program on Syzygies $\&$ Secants in 2024, where this paper was developed. DA was supported by the DFG under  the joint ANR-DFG program ``Positivity on K-trivial varieties'' (DFG Project Nr. 530132094.) and by the SFB-TRR 195. JP was supported by the National Research Foundation (NRF) funded by the Korea government (RS-2021-NR061320 and RS-2023-NR076427).

\section{Some background}\label{sec:background}

\noindent We start collecting some technical results as well as some algebro-combinatorial results about multisymmetric polynomials.

\subsection{Linear algebra for vector bundles}
We record a useful fact about alternating powers of vector bundles. 

\begin{lemma}\label{lem:seqbundles}
	Let $ 0 \to A \to B \to C \to 0$ be a short exact sequence of vector bundles on a scheme $X$ with $a:=\operatorname{rank} A, b:=\operatorname{rank} B, c:=\operatorname{rank} C$. For any $h\geq 0$ there are exact sequences
	\begin{align}
		\cdots  \longrightarrow  \wedge^{h-2}B \otimes S^2A \longrightarrow \wedge^{h-1}B\otimes A \longrightarrow  \wedge^h B  \longrightarrow  \wedge^h C  \longrightarrow  0 \label{eq:standardses} \\
		 \cdots \longrightarrow  \wedge^{h+2}B \otimes S^2C^{\vee}  \longrightarrow  \wedge^{h+1}B\otimes C^{\vee}  \longrightarrow \wedge^h B  \longrightarrow \wedge^{h-c} A\otimes \det C \longrightarrow  0 \label{eq:dualses}
	\end{align}
\end{lemma}
\begin{proof}
	The first sequence is standard (see \cite[Section V]{ABW82}). For the second, we can take the dual exact sequence and take $b-h$ alternating powers to obtain an exact sequence
	\[ \cdots  \longrightarrow  \wedge^{b-h-2}B^{\vee}\otimes S^2C^{\vee}  \longrightarrow \wedge^{b-h-1}B^{\vee} \otimes C^{\vee}  \longrightarrow  \wedge^{b-h}B^{\vee}  \longrightarrow  \wedge^{b-h}A^{\vee}  \longrightarrow  0.  \]
	Observe that
	\[ \wedge^{b-h-i}B^{\vee}\cong \wedge^{h+i}B\otimes \det B^{\vee}~~\text{ and }~~\wedge^{b-h}A^{\vee} \cong \wedge^{a-b+h}A \otimes \det A^{\vee} \cong \wedge^{h-c}A\otimes \det A^{\vee}.\]
	To conclude, we just need to recall that $\det B \cong \det A \otimes \det C$.
\end{proof}

\subsection{Sufficiently ample line bundles}
We briefly discuss some standard results on sufficiently ample line bundles. We also refer to \cite[Section 3]{BBF24+}. The most important result is perhaps the Fujita--Serre vanishing (see e.g., \cite[Theorem 1.4.35]{LazPosI}), which asserts that if $X$ is a projective scheme, $\mathcal{F}$ is a coherent sheaf on $X$ and $A$ is a sufficiently  ample line bundle on $X$, then $H^i(X, \mathcal{F} \otimes A)=0$ for $i>0$. It is easy to see that $H^0(X, \mathcal{F} \otimes A) \neq 0$ if and only if $\mathcal{F} \neq 0$. Recall that a line bundle $L$ on $X$ is said to be \emph{$d$-very ample} if for any finite subscheme $\xi\subseteq X$ of length $d+1$, the evaluation map
\[ \operatorname{ev}_{L,\xi}\colon H^0(X,L) \longrightarrow H^0(X,L\otimes \sO_{\xi}) \]
is surjective. If $d=0$, this is the same as being globally generated, and if $d=1$, this is the same as being very ample. If $L$ is globally generated, the \emph{kernel bundle} $M_L$ is defined by the short exact sequence
\[ 0 \longrightarrow M_L \longrightarrow H^0(X,L)\otimes \sO_X \xrightarrow{~\text{ev}_L~} L \longrightarrow 0. \]

\medskip

Consider the $k$-th Cartesian product $X^k$.
For $1 \leq i < j \leq k$, we denote by
$$
\Delta_{i,j}:=\{(y_1, \ldots, y_k) \in Y^k \mid y_i=y_j\} \subseteq Y^k
$$
the pairwise diagonal. It is easy to see that if $\pr_2 \colon X \times X \to X$ is the projection to the second factor, then $\pr_{2,*} ((\sO_X \boxtimes L)\otimes \mathcal{I}_{\Delta_{1,2}}) = M_L$ when $L$ is globally generated. 

\begin{lemma}\label{lem:M_L}
Let $Y$ be a projective scheme, $A_1, \ldots, A_k$ be globally generated line bundles on $Y$ such that $H^i(Y, A_j)=0$ for $i>0$ and $1 \leq j \leq k$, and $B$ be a vector bundle on $Y$. Then
$$
H^i(Y, M_{A_1} \otimes \cdots \otimes M_{A_k} \otimes B) \cong H^i(Y^{k+1},(A_1\boxtimes \cdots \boxtimes A_k \boxtimes B)\otimes (\mathcal{I}_{\Delta_{1,k+1}}\otimes \cdots \otimes \mathcal{I}_{\Delta_{k,k+1}}))
$$
for all $i \geq 0$. 
\end{lemma}

\begin{proof}
We use the arguments in Inamdar's paper \cite{Inamdar}. 
Let $\pr_{k+1} \colon Y^{k+1} \to Y$ be the projection to the last factor. Note that
       $$
       R^i \pr_{k+1,*} (A_1 \boxtimes \cdots \boxtimes A_k \boxtimes B) \otimes (\mathcal{I}_{\Delta_{1,k+1}}\otimes \cdots \otimes \mathcal{I}_{\Delta_{k,k+1}}) \cong \begin{cases} M_{A_1} \otimes \cdots \otimes M_{A_k} \otimes B & \text{for $i=0$} \\ 0 & \text{for $i>0$}. \end{cases}
       $$
Considering the Leray spectral sequence for $\pr_{k+1}$, we obtain the lemma. 
\end{proof}

\begin{lemma}\label{lem:suffamplelemma}
	Let $Y$ be a projective scheme and consider coherent sheaves $\mathcal{F},\mathcal{G}$  on $Y$ and  two line bundles $A,B$ on $Y$. 
	\begin{enumerate}
		\item A morphism $\mathcal{F}\to \mathcal{G}$ is surjective if and only if the induced map on global sections $H^0(X,\mathcal{F}\otimes A ) \to H^0(X,\mathcal{G}\otimes A)$ is surjective for $A$ sufficiently ample.
		\item If $A$ is sufficiently ample with respect to a fixed $m\geq 0$, then $A$ is $d$-very ample.
		\item The sheaf $\mathcal{F}\otimes A$ is globally generated for $A$ sufficiently ample.
		\item The multiplication map $H^0(Y,\mathcal{F}\otimes A) \otimes H^0(Y,\mathcal{G}\otimes B) \to H^0(Y,\mathcal{F}\otimes \mathcal{G}\otimes A\otimes B)$ is surjective if $A,B$ are sufficiently ample.
		\item Let $W$ be a projective variety, $\mathcal{H}$ be a coherent sheaf on $W\times Y$, and $L$ be a line bundle on $W$. Then the multiplication map $H^0(W\times Y,\mathcal{H}\otimes (L\boxtimes A)) \otimes H^0(Y,B) \to H^0(W\times Y, \mathcal{H}\otimes (L\boxtimes (A\otimes B)))$ is surjective if $A,B,L$ are sufficiently ample. 
		\item Let $f\colon Z\to Y$ be a morphism from a projective scheme $Z$, and $\mathcal{H}$ be a coherent sheaf on $Z$. For a fixed integer $i > 0$, we have that $H^i(Z,\mathcal{H}\otimes f^*A)=0$ for $A$ sufficienty ample if and only if $R^if_*\mathcal{H}=0$.
		\item $H^i(Y,M_{A_{1}}\otimes \cdots \otimes M_{A_n}\otimes B) = 0$ for all $i>0$ if $A_1,\ldots,A_n,B$ are sufficiently ample line bundles on $Y$.
		\item If $A$ is sufficiently ample, then it is very ample, $Y \subseteq  \nP(H^0(Y,A))=\nP^r$ is arithmetically normal (i.e., the restriction map $H^0(\nP^r, \sO_{\nP^r}(\ell))\to H^0(Y, \sO_Y(\ell))$ is surjective for every $\ell \geq 1$), and the ideal $I(Y, A)$ is generated by quadrics.
	\end{enumerate}
\end{lemma}

\begin{proof}
(1) This is essentially \cite[Lemma 5.3]{Ago22}. Considering the cokernel of the map $\mathcal{F} \to \mathcal{G}$ and applying Fujita--Serre vanishing, one can easily check the assertion.

\smallskip

\noindent (2) The case $d=1$ is proven in \cite[Proposition 3.6]{BBF24+}, and the general case is analogous. Here we give an alternative proof. First, take a very ample line bundle $H$ such that the embedding $Y \subseteq \nP (H^0(Y, H))$ has sufficiently large degree so that every length $d+1$ finite subscheme $\xi$ of $Y$ is contained in a finite subscheme $Z$ of $Y$ cut out by a fixed number of members of the complete linear system $\lvert H \rvert$. As the Koszul resolutions of such $Z$ are of the same form, Fujita--Serre vanishing implies that $\operatorname{ev}_{A,Z}$ is surjective whenever $A$ is sufficiently ample. This implies that $\operatorname{ev}_{A,\xi}$ is surjective for any finite subscheme $\xi \subseteq Y$ of length $m+1$

\smallskip

\noindent (3) Up to twisting $\mathcal{F}$ by a certain ample line bundle, we can assume that $\mathcal{F}$ is globally generated. Then it is enough to show that $A$ itself is globally generated if $A$ is sufficiently ample, and this is the case $d=0$ of point (2).

\smallskip

\noindent (4) Consider the diagonal $\Delta_{1,1} \subseteq Y^2$ and the surjective morphism $\mathcal{F}\boxtimes \mathcal{G} \to (\mathcal{F}\boxtimes \mathcal{G})\otimes \sO_{\Delta_{1,1}}$. Twisting by $A\boxtimes B$ and taking global sections we obtain the multiplication map that we care about. Then we just need to observe that if $A,B$ are sufficiently ample on $Y$ then $A\boxtimes B$ is sufficiently ample on $Y^2$, so that the conclusion follows from point (1).

\smallskip

\noindent (5) The proof of this point is analogous to the previous one. In the space $W\times Y\times Y$, we consider the surjective morphism $\pr_{W\times Y}^*\mathcal{H} \to \pr_{W\times Y}^*\mathcal{H}\otimes \pr_{Y\times Y}^*\mathcal{O}_{\Delta_{1,1}}$ and then we take global sections after tensoring with $L\boxtimes A\boxtimes B$.

\smallskip

\noindent (6) This is \cite[Lemma 4.3.10]{LazPosI}, but we give a proof for the reader's convenience.  Observe that for any $0\leq j\leq i$, we have $R^jf_*(\mathcal{H}\otimes f^*A) \cong R^jf_*\mathcal{H}\otimes A$ by the projection formula and if $A$ is sufficiently ample, then $H^{i-j}(Z,R^jf_*\mathcal{H}\otimes A)=0$ for all $j<i$ thanks to Fujita--Serre vanishing. By the Leray spectral sequence, we then see that $H^i(Z,\mathcal{H}\otimes f^*A) \cong H^0(Y,R^if_*\mathcal{H}\otimes A)$, and this latter group is zero for $A$ sufficiently ample if and only if $R^if_*\mathcal{H}=0$.

\smallskip

\noindent  (7) Thanks to point (2), we can assume that the $A_i$ are globally generated, so that the kernel bundles $M_{A_i}$ are well defined. By Lemma \ref{lem:M_L}, it is enough to show that
       \[ H^i(Y^{i+1},(A_1\boxtimes \cdots \boxtimes A_i\boxtimes B)\otimes (\mathcal{I}_{\Delta_{1,n+1}}\otimes \cdots \otimes \mathcal{I}_{\Delta_{n,n+1}})) = 0 \qquad \text{ for all } 1\leq i \leq n, \]
       and this is true by Fujita--Serre vanishing since $A_1,\ldots,A_n,B$ are sufficiently ample.

\smallskip

\noindent  (8) The fact that $A$ is very ample follows from point (2). It is then well-known (see e.g., \cite[Lemma 1.6]{EL1993}) that if $H^1(Y,A^{\otimes \ell}) = H^1(Y,M_A \otimes A^{\otimes \ell}) = H^1(X,\wedge^2 M_A\otimes A^{\otimes \ell}) = 0$ for all $\ell\geq 1$, then the embedding induced by $A$ is arithmetically normal and the ideal $I(Y,A)$ is generated by quadrics. The conclusion follows then from point (7) since $\wedge^2 M_A$ is a direct summand of $M_A^{\otimes 2}$ in characteristic different from two.
\end{proof}

We record an effective version of Lemma \ref{lem:suffamplelemma} (2) if $Y$ is smooth:

\begin{lemma}\label{lem:effhigherveryampleness}
	Let $Y$ be a smooth projective variety of dimension $n$, and $H, M$ be line bundles on $Y$ such that $H$ is very ample and $M$ is nef.  Then $L :=\omega_Y\otimes H^{\otimes j} \otimes M$ is $m$-very ample for all $j\geq n+1+m$.
\end{lemma}
\begin{proof}
Note that $\omega_Y \otimes H^{\otimes (n+1)} \otimes M$ is $0$-very ample (see e.g., \cite[Example 1.8.23]{LazPosI}) and $H$ is $1$-very ample. Thus  \cite[Theorem 1.1]{HTT} shows that $L$ is $m$-very ample as soon as $j \geq n+1+m$.
\end{proof}

\subsection{Multisymmetric polynomials}
We start with some notation: let $V$ be a $\kk$-linear representation of $\mathfrak{S}_d$ and let  $\varepsilon_d\colon V \to \{\pm 1\}$ be the alternating character of $\mathfrak{S}_d$. Then we denote by $V\otimes \varepsilon_d$ the the same vector space $V$, but with the action of $\mathfrak{S}_d$ twisted by $\varepsilon_d$ (\footnote{This isomorphic to the tensor product of $V$ with the one-dimensional representation corresponding to $\mathfrak{S}_d$, but we want to think of $V$ and $V\otimes \varepsilon_d$ has having the same base set}). We have
\[ V^{\mathfrak{S}_d} = \{ v\in V \,|\, \sigma\cdot v = v \}~~\text{ and }~~ (V\otimes \varepsilon_d)^{\mathfrak{S}_d} = \{ v\in V \,|\, \sigma\cdot v = \varepsilon_d(\sigma)v \} \]
Furthermore, we can also define $V\otimes \varepsilon_d^{\ell}$ for any $\ell\in \nZ$ in an analogous way: this is isomorphic to $V$ or $V\otimes \varepsilon_d$ according to whether $\ell$ is even or odd. 
\medskip

Set now $X:= \nA^n = \Spec \kk[x_1,\ldots,x_n]$, and take $X^d = \Spec  \kk[x_{ij}]$ for $1\leq i\leq d$ and $1\leq j \leq n$. The closed points of $X^d$ can be thought of as $n\times d$ matrices
\[ 
\begin{pmatrix}
	x_{11} & x_{12} & x_{13} & \cdots & x_{1d} \\
	x_{21} & x_{22} & x_{23} & \cdots & x_{2d} \\
	\vdots & \vdots & \vdots & \ddots & \vdots \\
	x_{n1} & x_{n2} & x_{n3} & \cdots & x_{nd}	 
\end{pmatrix} 
\]
where each column corresponds to a point in $\nA^n$. This space has a natural action of the symmetric group $\mathfrak{S}_{d}$ and the ring of multisymmetric functions is $\kk[x_{ij}]^{\mathfrak{S}_d}$, while the space of multialternating functions is  $(\kk[x_{ij}]\otimes \varepsilon_d)^{\mathfrak{S}_d}$. This generates an ideal
\[ J_d := ((\kk[x_{ij}]\otimes \varepsilon_d)^{\mathfrak{S}_d}) \subseteq \kk[x_{ij}] \] 
that it generates inside $\kk[x_{ij}]$.
Since an alternating polynomial vanishes if two columns of the above matrix are equal, it follows that $J_d \subseteq I(\Delta_d)$, where $I(\Delta_d) \subseteq \nC[x_{ij}]$ is the ideal of the big diagonal 
$$
\Delta_d := \bigcup_{1 \leq i < j \leq d} \Delta_{i,j} \subseteq X^d. 
$$
For any $\ell\geq 1$ we have natural multiplication maps
\begin{equation}\label{eq:mapalt} 
	m^{\ell}\colon ((\kk[x_{ij}]\otimes \varepsilon_d)^{\mathfrak{S}_d})^{\otimes \ell} \longrightarrow (I(\Delta_d)^{\ell}\otimes \varepsilon_d^{\ell})^{\mathfrak{S}_d}; \quad (f_1\otimes \dots \otimes f_{\ell}) \longmapsto f_1 \dots f_{\ell} 
\end{equation}
of $\kk[x_{ij}]^{\mathfrak{S}_d}$-modules. What we need is the following:

\begin{proposition}\label{prop:mappoly}
	The image of the map $m^{\ell}$ of \eqref{eq:mapalt} is:
	\[ \operatorname{Im} m^{\ell} = \left(J_d^{\ell} \otimes \varepsilon_d^{\ell}\right)^{\mathfrak{S}_d} = \left(J_d^{\ell-1} \otimes \varepsilon_d^{\ell}\right)^{\mathfrak{S}_d}. \]
\end{proposition}
\begin{proof}
By construction it holds that $\operatorname{Im} m^{\ell} \subseteq \left(J_d^{\ell} \otimes \varepsilon_d^{\ell}\right)^{\mathfrak{S}_d} \subseteq \left(J_d^{\ell-1} \otimes \varepsilon_d^{\ell}\right)^{\mathfrak{S}_d}$. We will now show that $\left(J_d^{\ell-1} \otimes \varepsilon_d^{\ell}\right)^{\mathfrak{S}_d}\subseteq \operatorname{Im} m^{\ell}$. We know that $\left(J_d^{\ell-1} \otimes \varepsilon_d^{\ell}\right)^{\mathfrak{S}_d}$ is the image of the projection
	\begin{align*}  
		\pi^{\mathfrak{S}_d}_d&\colon \left(J_d^{\ell-1} \otimes \varepsilon_d^{\ell}\right) \longrightarrow \left(J_d^{\ell-1} \otimes \varepsilon_d^{\ell}\right)^{\mathfrak{S}_{d}};\quad f\longmapsto \frac{1}{d!}\sum_{\sigma\in \mathfrak{S}_d} \varepsilon_{d}^{\ell}(\sigma)\sigma(f).
	\end{align*}
Observe that an element in $J_d^{\ell-1}\otimes \varepsilon_d^{\ell}$ is a sum of elements of the form $g\cdot f_1 \dots  f_{\ell-1}$ with $g\in \kk[x_{ij}],f_i \in (\kk[x_{ij}]\otimes \varepsilon_d)^{\mathfrak{S}_d}$. For any such element it is easy to see that
\begin{align*} 
\pi_d^{\mathfrak{S}_d}(gf_1 \dots  f_{\ell-1}) &= \left( \frac{1}{d!}\sum_{\sigma\in \mathfrak{S}_d} \varepsilon_d(\sigma)^{2\ell-1}\cdot \sigma(g) \right)\cdot f_1\dots f_{\ell-1} \\
&  = \left( \frac{1}{d!}\sum_{\sigma\in \mathfrak{S}_d} \varepsilon_d(\sigma)\cdot \sigma(g) \right)\cdot f_1\dots f_{\ell-1} 
\end{align*}
and since the first factor in the last product belongs to $(\kk[x_{ij}]\otimes \varepsilon_d)^{\mathfrak{S}_d}$, we see that $\pi_d^{\mathfrak{S}_d}(gf_1 \dots  f_{\ell-1}) \in \operatorname{Im} m_d$.
\end{proof}

\begin{remark}\label{rmk:equalityideals}
	Assume that $J_d = I(\Delta_d)$: then it follows that $(J_d^{\ell}\otimes \varepsilon_d^{\ell}) ^{\mathfrak{S}_d}=(I(\Delta_d)^{\ell}\otimes \varepsilon_d^{\ell})^{\mathfrak{S}_d}$ for all $\ell\geq 1$.
	It is straightforward to show that $J_d = I(\Delta_d)$ if $n=1$ or $d=1$. If instead $n=2$ and $d$ is arbitrary, then the equality $J_d = I(\Delta_d)$ still holds, but it is a highly nontrivial theorem due to Haiman \cite[Corollary 3.8.3]{Haiman} (see also \cite[Theorem 1.1]{Haimanlectures}). We could also check the equality  via computer algebra if $n=d=3$. For all other cases, this is still open to the best knowledge of the authors. In the next proposition, we prove a weaker result, which will be however essential for Theorem \ref{thm:SS}. 
\end{remark}

\begin{proposition}\label{prop:J3}
We have $J_3^{\mathfrak{S}_3} = I_{\Delta_3}^{\mathfrak{S}_3}$.
\end{proposition}

\begin{proof}
	For  $1\leq a<b \leq 3$, let $I_{\Delta_{a,b}} \subseteq \kk[x_{ij}]$ be the ideal of the pairwise diagonal $\Delta_{a,b} \subseteq (\nA^n)^3$. It is proven in \cite[Corollary 4.21]{ScalaDiag} that $I_{\Delta_3}^{\mathfrak{S}_3} = \left( I_{\Delta_{1,2}}\cdot I_{\Delta_{1,3}} \cdot I_{\Delta_{2,3}} \right)^{\mathfrak{S}_3}$. If we can prove that $( I_{\Delta_{1,2}}\cdot I_{\Delta_{1,3}} \cdot I_{\Delta_{2,3}}) \subseteq J_3$, we are done. The product ideal is generated by elements of the form
	\[ (x_{1a}-x_{2a})\cdot (x_{1b}-x_{3b})\cdot (x_{2c}-x_{3c}) \qquad \text{ for } \{a,b,c\} \in \{1,\ldots,n\}. \]
By symmetry, if we can prove that all of these are contained in $J_3$ when $n=3$, then it follows in the general case. Finally, in the case $n=3$ we checked that $J_3=I(\Delta_3)$ with both computer algebra programs \texttt{Macaulay2} \cite{M2} and \texttt{OSCAR} \cite{OSCAR}. 
The code can be found at the arXiv version of this paper and as an online tutorial \cite{OscarTutorial} in the case of \texttt{OSCAR}\footnote{We recommend the interested reader to refer to the online tutorial \cite{OscarTutorial}, since this will be kept more up to date.}.
\end{proof}

\section{The yoga of tautological bundles on Hilbert schemes of points}\label{sec:yoga} 

\noindent In this section, we first review Hilbert schemes of points on smooth varieties and tautological bundles following \cite{Ago22, Ago24, CKP23+, CLPS25+, ENP, NP24+}, and then, we show some technical results, which will be crucial for the proofs of the main theorems of the paper. In particular, we prove Proposition \ref{prop:HC=blow-up} from the introduction and some essential cohomology vanishing results: Proposition \ref{prop:cohvanishing} and Proposition {prop:cohvanishingcurves}.

\subsection{Hilbert schemes of points on smooth varieties} 
Let $X$ be a smooth projective variety of dimension $n$. For any $d\geq 1$, we denote by $X^{(d)} = X^d/\mathfrak{S}_d$ the $d$-th \emph{symmetric product} of $X$ with the quotient map 
\[ q_d\colon X^d \longrightarrow X^{(d)}. \]
Recall that we denote instead by $X^{[d]}$ the Hilbert scheme of $d$ points on $X$. It is well known that the Hilbert scheme $X^{[d]}$ is smooth if and only if $n\leq 2$ or $d\leq 3$. In this section, we always assume that this is the case unless otherwise specified.
The \emph{Hilbert--Chow morphism}
\[ h_d \colon X^{[d]} \longrightarrow X^{(d)} \]
sends a finite scheme $\xi\in X^{[d]}$ to the points on its support counted with multiplicity. The Hilbert--Chow morphism is an isomorphism if $d=1$ or if $X=C$ is a curve, but in all other cases it has positive-dimensional fibers. The locus of non-reduced schemes
\[ E_d = \{\xi \in X^{[d]} \,|\, \xi \text{ is non-reduced }\}  \subseteq X^{[d]}\]
is a divisor in $X^{[d]}$ and it is precisely the exceptional divisor of the Hilbert--Chow morphism.  

\medskip

When $1\leq d_1\leq d_2$, we also consider the \emph{nested Hilbert scheme}
\[ X^{[d_1,d_2]} = \{ (\xi_1,\xi_2) \in X^{[d_1]}\times X^{[d_2]} \,|\, \xi_1\subseteq \xi_2 \} \subseteq X^{[d_1]} \times X^{[d_2]} \]
with the two induced maps
\[ \tau\colon X^{[d_1,d_2]} \longrightarrow X^{[d_1]}~~\text{ and }~~ \rho\colon X^{[d_1,d_2]}\longrightarrow X^{[d_2]}. \]
In particular, the nested Hilbert scheme $X^{[1,d]}$ is the universal family over $X^{[d]}$, meaning that the fiber of the map $\rho\colon X^{[1,d]} \to X^{[d]}$ over $\xi\in X^{[d]}$ is $\rho^{-1}(\xi) = \xi \times \{\xi\} \subseteq X\times \{\xi\}$. Thus $\rho\colon X^{[1,d]} \to X^{[d]}$ is finite and flat of degree $d$. Sometimes we will also use the notation $\mathcal{Z}_d = X^{[1,d]}$ for the universal family. If instead $2\leq d$, and if $(\xi_1,\xi_2)\in X^{[d-1,d]}$, then the scheme-theoretic difference $\xi_2\setminus \xi_1$ consists of a single point, so that we have a well-defined residual morphism
\[ \operatorname{res}\colon X^{[d-1,d]} \longrightarrow X; \quad (\xi_1,\xi_2) \longmapsto \xi_2\setminus \xi_1.  \]
It is well known that the nested Hilbert scheme is smooth precisely in the following cases:
	\[ n=1, \quad n=2, d_2=d_1+1, \quad d_1=1,d_2=2, \quad d_1=2,d_2=3.\]
Since these are the cases that we will use in the following, we always assume that we are in this situation.
The combination 
\[ (\operatorname{res}, \tau )\colon X^{[d-1,d]} \longrightarrow X \times X^{[d-1]} \]
of two maps $\res$ and $\tau$  realizes the nested Hilbert scheme $X^{[d-1,d]}$ as the blow-up along the universal family $\mathcal{Z}_m \subseteq X \times X^{[d-1]}$. We denote by $F_{d-1}\subseteq X^{[d-1,d]}$ the corresponding exceptional divisor. We have
\[ F_{d-1} = \{ (\xi_1,\xi_2) \in X^{[d-1,d]} \,|\, \operatorname{Supp}(\xi_1) = \operatorname{Supp}(\xi_2) \}. \] 
For more details on Hilbert schemes of points and nested Hilbert schemes, we refer to \cite[Section 1]{CLPS25+}.

\subsection{Tautological bundles on Hilbert schemes of points} 
Let now $L$ be any line bundle on $X$. Consider the universal family $\mathcal{Z}_d\subseteq X \times X^{[d]}$ with the map $\rho \colon \mathcal{Z}_d \to X^{[d]}$. The sheaf
\[ E_{d,L} := \rho_*((\sO_{X^{[d]}} \boxtimes L  ) \otimes \sO_{\mathcal{Z}_d}) \]
is a vector bundle of rank $d$ on $X^{[d]}$, called the \emph{tautological bundle}. Its fiber over $\xi\in X^{[d]}$ is naturally identified with $H^0(X,L\otimes \sO_{\xi})$. Taking the determinant, we obtain a line bundle
\[ N_{d,L} := \wedge^d E_{d,M}. \]
In particular, one defines a divisor $\delta_d$ on $X^{[d]}$ such that $\sO_{X^{[d]}}(-\delta_d) := N_{d,\sO_X}$. It holds that $\sO_{X^{[d]}}(2\delta_d) \cong \sO_{X^{[d]}}(E_d)$. Starting with the line bundle $L$, we can take a line bundle $L^{\boxtimes d}$ on $X^d$. This can be made into a $\mathfrak{S}_d$-bundle in two different ways: one with the natural action and another one with the natural action twisted by the alternating character $\varepsilon_d$ of $\mathfrak{S}_d$. We denote the two $\mathfrak{S}_d$-bundles as
\[ L^{\boxtimes d}~~\text{ and }~~ L^{\boxtimes d}\otimes \varepsilon_d, \] 
and we obtain two sheaves on the symmetric product $X^{(d)}$ as
\[ S_{d,L} := q_{d,*}^{\mathfrak{S}_d}(L^{\boxtimes d})~~\text{ and }~~ \mathcal{N}_{d,L} := q_{d,*}^{\mathfrak{S}_d}(L^{\boxtimes d} \otimes \varepsilon_d). \]
The first one is a line bundle, but the second one is torsion-free of rank one (we use curly letter for sheaves that are not bundles). By pulling back $S_{d,L}$ via the Hilbert--Chow morphism $h_m \colon X^{[d]} \to X^{(d)}$, we obtain a line bundle
\[ T_{d,L} := h^*S_{d,L}\]
on $X^{[d]}$. 
We also define 
\[ A_{d,L} := T_{d,L}(-2\delta_m) \cong T_{d,L}(-E_d). \]
If $L_1,L_2$ are two line bundles on $X$, it holds that
\begin{align*} 
S_{d,L_1\otimes L_2} &\cong S_{d,L_1}\otimes S_{d,L_2}, &
\mathcal{N}_{d,L_1\otimes L_2} &\cong \mathcal{N}_{d,L_1}\otimes S_{d,L_2},\\
T_{d,L_1\otimes L_2} &\cong T_{d,L_1}\otimes T_{d,L_2}, &
N_{d,L_1\otimes L_2}&\cong N_{d,L_1}\otimes T_{d,L_2}.
\end{align*}
and in particular, for any line bundle $L$, it holds that.
\[ \mathcal{N}_{d,L}\cong  S_{d,L} \otimes \mathcal{N}_{d,\sO_X}~~\text{ and }~~N_{d,L}\cong T_{d,L}\otimes N_{d,\sO_X} \cong T_{d,L}(-\delta_d). \]
Furthermore, $q_d^*(S_{d,L})\cong L^{\boxtimes d}$ so that, if $L$ is ample on $X$, then $S_{d,L}$ is ample on $X^{(d)}$. This means that if $L$ is sufficiently ample on $X$, then $S_{d,L}$ is sufficiently ample on $X^{(d)}$.

\medskip

An useful observation is that there are isomorphisms of sheaves on $X^{(d)}:$
\begin{equation}\label{eq:isopushforward} 
	h_*T_{d,L} \cong S_{d,L},\quad h_*N_{d,L} \cong \mathcal{N}_{d,L}, \quad h_*A_{d,L} \cong  S_{d,L}\otimes  \mathcal{I}_{\Delta_d} 
\end{equation}
Indeed, the first two follow from \cite[Discussion below Lemma 2.3]{CLPS25+}, and the last one follows from the projection formula applied to $A_{d,L} = h^*S_{d,L}\otimes \sO_{X^{[d]}}(-E_{d})$ and Proposition \ref{prop:idealXmsmooth}, which will be shown below.
The global sections of these bundles can be explicitly computed. From \cite[Discussion below Lemma 2.3]{CLPS25+}, we see that there are canonical isomorphisms
\begin{align*} 
 H^0(X^{[d]},T_{d,L}) \cong H^0(X^{(d)}&,S_{d,L}) \cong S^dH^0(X,L), \\
  H^0(X^{[d]},N_{d,L}) \cong H^0(X^{(d)}&,\mathcal{N}_{d,L}) \cong \wedge^d H^0(X,L),\\
  H^0(X^{[d]},E_{d,L}) &\cong H^0(X,L).
\end{align*}
Furthermore, if $d\geq 2$ and $L$ is sufficiently ample, then \cite[Proposition 6.5]{CLPS25+} gives an isomorphism
\[ H^0(X^{[d]},A_{d,L})) \cong I(\Sigma_{d-2}(X,L))_d. \]
Note that the tautological bundle $E_{d,L}$ is globally generated if and only if $L$ is $(d-1)$-very ample. In this case, by pushing forward the exact sequence
\[ 0\longrightarrow \mathcal{I}_{\mathcal{Z}_d} \otimes (\sO_{X^{[d]}}\boxtimes L) \longrightarrow \sO_{X^{[d]}}\boxtimes L \longrightarrow   \sO_{\mathcal{Z}_d} \otimes (\sO_{X^{[d]}}\boxtimes L) \longrightarrow 0 \]
to $X^{[d]}$, we obtain an exact sequence
\[ 0 \longrightarrow M_{d,L} \longrightarrow H^0(X,L)\otimes \sO_{X^{[d]}} \longrightarrow E_{d,L} \longrightarrow 0 \]
that defines the \emph{kernel bundle} $M_{d,L}$ on $X^{[d]}$. By Lemma \ref{lem:suffamplelemma} (2), $E_{d,L}$ is globally generated and the kernel bundle $M_{d,L}$ is well defined whenever $L$ is sufficiently ample. 
For more details on basic properties of tautological bundles and kernel bundles, we refer to \cite[Section 2]{CLPS25+}.

\subsection{Hilbert--Chow morphism as a blow-up}	
Consider the \emph{big diagonals} in the Cartesian and symmetric products
\[ \Delta_d := \bigcup_{1 \leq i < j \leq d} \Delta_{i,j} \subseteq X^d~~\text{ and }~~\Delta_{(d)} := q_d(\Delta_d) \subseteq X^{(d).} \] 
We have the equality of sheaves on $X^{(d)}$:
\begin{equation}
	\mathcal{I}_{\Delta_{(d)}} = q_{d,*}^{\mathfrak{S}_d}(\mathcal{I}_{\Delta_d}).
\end{equation}
Indeed, the regular functions on $X^{(d)}$ that vanish on $\Delta_{(d)}$ are precisely the $\mathfrak{S}_d$-invariant regular functions on $X^d$ that vanish on $\Delta_d$. Following the discussion on multisymmetric polynomials in Section \ref{sec:background}, another natural ideal sheaf on $X^d$ is the one $\mathcal{J}_d \subseteq \mathcal{I}_{\Delta_d}  \subseteq \sO_{X^d}$ which is locally generated by alternating regular functions. We then obtain another ideal sheaf on $X^{(d)}$:
\[ \mathcal{J}_{(d)} := q_{d,*}^{\mathfrak{S}_d}(\mathcal{J}_d) \subseteq \mathcal{I}_{\Delta_{(d)}}. \]
By Proposition \ref{prop:mappoly}, this ideal is precisely the image of the multiplication map
\[ \mathcal{N}_{d,\sO_X} \otimes \mathcal{N}_{d,\sO_X} \longrightarrow \mathcal{I}_{\Delta_{(d)}}. \]
A priori $\mathcal{J}_{(d)}$ and $\mathcal{I}_{\Delta_{(d)}}$ might differ, but if the Hilbert scheme $X^{[d]}$ is smooth, then we show that they coincide in the following proposition. The statement is significant when $n=2$, where it was proven by Haiman \cite{Haiman}, or when $d=3$, where it follows from Proposition \ref{prop:J3}.

\begin{proposition}\label{prop:idealXmsmooth}
	Assume that $X^{[d]}$ is smooth, or, equivalently, that $n\leq 2$  or $d\leq 3$. Then the multiplication map
	\[ \mathcal{N}_{d,\sO_X} \otimes \mathcal{N}_{d,\sO_X} \longrightarrow \mathcal{I}_{\Delta_{(d)}} \]
	 of sheaves on $X^{(d)}$ is surjective, and there is an equality
	\[ \mathcal{J}_{(d)} = h_*\sO_{X^{[d]}}(-E_d) =   \mathcal{I}_{\Delta_{(d)}} \]
	of sheaves on $X^{(d)}$.
	In particular, the Hilbert--Chow morphism $h_d \colon X^{[d]} \to X^{(d)}$ realizes $X^{[d]}$ as the blow-up of $X^{(d)}$ along the big diagonal $\Delta_{(d)}$ with its reduced structure.
\end{proposition}
\begin{proof}
	 We use a result of Ekedahl and Skjelnes \cite{ESGood}. A combination of the remark at the beginning of 
\cite[Section 3.3]{ESGood} and  \cite[Proposition 3.4 and Theorem 7.7]{ESGood} shows that the Hilbert--Chow morphism $h_d \colon X^{[d]} \to X^{(d)}$ is the blow-up along the ideal sheaf that is the image of the map $\mathcal{N}_{d,\sO_X} \otimes \mathcal{N}_{d,\sO_X} \rightarrow \mathcal{I}_{\Delta_{(d)}}$.  Recall that this ideal sheaf is precisely $\mathcal{J}_{(d)}$ by Proposition \ref{prop:mappoly}. Hence we have $\mathcal{J}_{(d)} \subseteq h_*\sO_{X^{[d]}}(-E_d) \subseteq \mathcal{I}_{\Delta_{(d)}}$. At this point, Remark \ref{rmk:equalityideals} for $n\leq 2$ or $d\leq 2$ and  Proposition \ref{prop:J3} for $d=3$ imply that $\mathcal{J}_{(d)} = \mathcal{I}_{\Delta_{(d)}}$.
\end{proof}

\begin{remark} 
	Let for a moment $n,d$ be general. The the \emph{smoothable component} of $X^{[d]}$ is the schematic closure $X^{[d]}_{\operatorname{sm}} = \overline{(X^{[d]}\setminus E_d)}$. Using the results of \cite{ESGood} and Proposition \ref{prop:mappoly} as in the above proof, one sees that the restriction of the Hilbert--Chow morphism
	\[ h_d \colon X^{[d]}_{\operatorname{sm}} \longrightarrow X^{(d)}\]
	is  the blow-up along the ideal sheaf $\mathcal{J}_{(d)}$ and the intersection $E_{d,\operatorname{sm}}:=E_d \cap X^{[d]}_{\operatorname{sm}}$ is the corresponding exceptional divisor. Hence, there are always containments
	\[ \mathcal{J}_{(d)} \subseteq h_*\sO_{X^{[d]}_{\operatorname{sm}}}(-E_{d,\operatorname{sm}}) \subseteq \mathcal{I}_{\Delta_{(d)}}.  \]
	More precisely, $h_*\sO_{X^{[d]}_{\operatorname{sm}}}(-E_{d,\operatorname{sm}})$ is the integral closure of $\mathcal{J}_{(d)}$ inside $\sO_{X^{(d)}}$ (see \cite[Definition 9.6.2]{LazPosI}).
\end{remark}

The proof of Proposition \ref{prop:HC=blow-up} just collects together facts that we proved up to now.

\begin{proof}[Proof of Proposition  \ref{prop:HC=blow-up}]
	The first statement follows from the discussion in Remark \ref{rmk:equalityideals} and Proposition \ref{prop:J3}. The second statement is in Proposition \ref{prop:idealXmsmooth}.
\end{proof}
	
\subsection{Some cohomology computation}
When $d \geq 2$, the nested Hilbert scheme $X^{[d-1,d]}$ with the maps $\tau\colon X^{[d-1,d]}\to X^{[d]},\rho\colon X^{[d-1,d]}\to X^{[d-1]}$, and $\res\colon X^{[d-1,d]}\to X$ is useful to reduce questions on $X^{[d]}$ to questions on $X^{[d-1]}\times X$. A first important fact is that for any line bundle $L$ on $X$, there is a short exact sequence
\begin{equation}\label{eq:seqpullbackE}
	0 \longrightarrow \res^*L(-F_{d-1}) \longrightarrow \rho^*E_{d,L} \longrightarrow \tau^*E_{d-1,L} \longrightarrow 0, 
\end{equation}
and in particular we see that
\begin{equation}\label{eq:pullbackNT}
	\rho^*N_{d,L} \cong ( \res^*L \otimes \tau^{*}N_{d-1,L})(-F_{d-1})~~\text{ and }~~ \rho^*T_{d,L} \cong \res^*L  \otimes \tau^*T_{d-1,L}.
\end{equation} 
If $L$ is $(d-1)$-very ample, then there is also another exact sequence
\begin{equation}\label{eq:seqpullbackM}
	0 \longrightarrow \rho^*M_{d,L} \longrightarrow \tau^*M_{d,L} \longrightarrow \res^*L(-F_{d-1}) \longrightarrow 0
\end{equation}
Finally, since we are assuming that $X^{[d-1,d]}$ is smooth, \cite[Lemma 2.7]{CLPS25+} says that 
\begin{equation}\label{eq:directsummand}
	\sO_{X^{[d]}} \text{ is a direct summand of } \rho_*\sO_{X^{[d-1,d]}}(F_{d-1}).
\end{equation}
Now, we can prove one of main technical results that is going to be essential for later.

\begin{proposition}\label{prop:cohvanishing}
	Let $X$ be a smooth projective variety of dimension $n$, and fix an integer $d$ such that $n\leq 2$ or $d\leq 2$. If $B$ is an arbitrary line bundle on $X$ and $A$ is a sufficiently ample line bundle on $X$, then 
	\[ H^i(X^{[d]},S^iE_{d,B}^{\vee} \otimes T_{d,A}) =H^i(X^{[d]},S^iE_{d,B}^{\vee} \otimes N_{d,A}) = 0 \quad \text{ for all } i>0. \]
	Equivalently, it holds that
	\[R^ih_*(S^iE_{d,B}^{\vee})=R^ih_*(S^iE_{d,B}^{\vee}\otimes \sO_{X^{[d]}}(-\delta_{d})) = 0 \quad \text{ for all } i>0, \]
	where $h = h_d \colon X^{[d]} \to X^{(d)}$ is the Hilbert--Chow morphism.
\end{proposition}
\begin{proof}
	We first show that the two statements are equivalent. We know that $N_{d,A}\cong \sO(-\delta_d)\otimes h^*S_{d,A}$, and furthermore, if $A$ is sufficiently ample on $X$, then $S_{d,A}$ is sufficiently ample on $X^{(d)}$. Then the equivalence of the statements follows from Lemma \ref{lem:suffamplelemma} (6). Now, we check the vanishing of the higher direct images. If $n=1$ or $d=1$, then the Hilbert--Chow morphism $h_d \colon X^{[d]} \to X^{(d)}$ is an isomorphism and it has no higher direct images. If $n=2$, then the vanishing was proven in \cite[Proposition 4.9]{CLPS25+}. We can now suppose that $n$ is arbitrary and $d=2$. Since the Hilbert--Chow morphism $h=h_2 \colon X^{[2]} \to X^{(2)}$ has fibers of dimension at most $n-1$, we need to show that
	\[ R^ih_*(S^iE_{d,B}^{\vee})=R^i h_*(S^iE_{2,B}^{\vee} \otimes \sO_{X^{[2]}}(-\delta_{2})) = 0 \qquad \text{ for } 1\leq i \leq n-1 .  \]
Consider the nested Hilbert scheme $X^{[1,2]}$ with the maps $\tau\colon X^{[1,2]} \to X^{[1]}$, $\rho\colon X^{[1,2]} \to X^{[2]}$, and $\res\colon X^{[1,2]} \to X$. Notice that $\rho \sO_{X^{[1,2]}} \cong \sO_{X^{[2]}} \oplus \sO_{X^{[2]}}(-\delta_2)$. Thus the previous vanishings are implied by
$$
		R^ih_*(\rho_*(\rho^* S^i E^{\vee}_{2,B})) = R^i(h \circ \rho)_* (\rho^* S^i E^{\vee}_{2,B})=0
\quad \text{ for } 1\leq i \leq n-1 ,
$$
	where the first equality follows from the fact that $\rho\colon X^{[1,2]} \to X^{[2]}$ is finite. Now we observe that  $(h\circ \rho) = q_2\circ (\res, \tau)$, where $q_2\colon X\times X \to X^{(2)}$ is the usual quotient by $\mathfrak{S}_2$. Since $q_2$ is finite, we see that the above vanishings are equivalent to
	\[ R^i(\res, \tau)_*(\rho^* S^i E^{\vee}_{2,B})=0 \quad \text{ for } 1\leq i \leq n-1.\] 
	 At this point, we see from \eqref{eq:seqpullbackE} that we have a short exact sequence
	\[ 0 \longrightarrow \tau^*E_{1,B}^{\vee} \longrightarrow \rho^*E^{\vee}_{2,B} \longrightarrow \res^*B^{-1}(F_1)  \longrightarrow 0 \]
	and taking the induced filtration on the symmetric product, we finally see that our vanishings are implied by 
	\begin{multline*} 
	R^i(\res,\tau)_*(\res^*B^{-j} \otimes \tau^*S^{i-j}E_{1,B}^{\vee} \otimes \sO_{X^{[1,2]}}( j F_1))\\
		 = (B^{-j} \boxtimes S^{i-j}E_{1,B}^{\vee}) \otimes R^i(\res,\tau)_*\sO_{X^{[1,2]}}( j F_1) = 0 \quad \text{ for } 1\leq i\leq n-1 \text{ and } 0\leq j\leq i.
	\end{multline*}
	Notice here that $E_{1,B}=B$ on $X$. 
	Then the last vanishing follows from the fact that $(\res,\tau)\colon X^{[1,2]} \to X \times X$ is the blow-up along the smooth subvariety $\mathcal{Z}_1 = \Delta_{1,1}$ of codimension  $n$, and $F_1$ is its exceptional divisor,  so that $R^i(\res,\tau)_*\sO_{X^{[1,2]}}(jF_1) = 0$ for all $i>0$ and $0\leq j \leq n-1$.
	\end{proof}

\subsection{The case of curves}
We assume that $X=C$ is a smooth projective curve of genus $g$. In this case, hings are much simpler since the Hilbert--Chow morphism $h_d \colon C^{[d]} \to C^{(d)}$ is an isomorphism. Furthermore, the map
\[ (\res,\tau)\colon C^{[d-1,d]} \longrightarrow C \times C^{[d-1]}\]
is also an isomorphism, and we see that $\rho\circ (\res,\tau)^{-1}\colon C \times C^{[d-1]} \to C^{[d]}$ is simply the addition map $\sigma_{d-1}\colon C \times C^{[d-1]} \to C^{[d]}$. If  $A$ is a $(d-1)$-very ample line bundle on $C$, the exact sequence \eqref{eq:seqpullbackM} on $C^{[d-1,d]}$ becomes an exact sequence
\begin{equation}\label{eq:pullbackMcurve} 
0 \longrightarrow \sigma_{d-1}^*M_{d,B} \longrightarrow \sO_C \boxtimes M_{d-1,B} \longrightarrow (B \boxtimes \sO_{C_{d-1}})(-\mathcal{Z}_{d-1}) \longrightarrow 0,
\end{equation}
where we notice that for curves the universal family $\mathcal{Z}_{d-1}\subseteq C \times C^{[d-1]}$ is a prime divisor isomorphic to $C \times C^{[d-2]}$. 

\medskip

An important feature of the curve case is that we can use the ``lifting technique'' of \cite{CKP23+, NP24+} to compute some cohomology groups on the Hilbert scheme of points on curves.

\begin{proposition}\label{prop:lifting}
	Let $A,B$ be two nonspecial line bundles on $C$ such that $A$ is $(a-1)$-very ample and $B$ is $(b-1)$-very ample. Then there are isomorphisms
	\[ H^i(C^{[a]},\wedge^b M_{a,A}\otimes S_{a,B})\cong H^i(C^{[b]},S^aM_{b,B}\otimes N_{b,A}) \quad \text{ for all } i \geq 0. \]
\end{proposition}
\begin{proof}
   Via the arguments in \cite[Lemma 3.3]{CKP23+} and \cite[Lemma 2.3]{NP24+}, one can see that both sides are isomorphic to $H^i(C^{[a]}\times C^{[b]} , (N_{a,A}\boxtimes S_{b,B})(-D_{a,b}))$, where $D_{a,b}$ is the divisor on $C^{[a]}\times C^{[b]}$ given by
   \[ D_{a,b} := \{(\xi,\xi')\in C^{[a]}\times C^{[b]} \,|\, \xi\cap\xi'\ne \emptyset \}. \qedhere\]
\end{proof} 

With this, we can prove the key technical statement that we are going to use for Theorem \ref{thm:EKS}. The statement should be compared to  \cite[Theorem 1.3]{NP24+}.

\begin{proposition}\label{prop:cohvanishingcurves}
	Let $A,B$ be two nonspecial and $(k+1)$-very ample line bundles on $C$ such that
	\[ h^1(C, B\otimes A^{-1}) \leq r(A)-k-2, \]
	where $r(A):=h^0(C,A)-1$. Then, for all $1\leq i \leq k+2$,  it holds that 
	\[ H^i(C^{[k+2]},S^iM_{k+2,B}\otimes N_{k+2,A}) = H^i(C^{[i]},\wedge^{k+2}M_{i,A}\otimes S_{i,B}) = 0.  \]
\end{proposition}
\begin{proof} 
Since  $A,B$ are both nonspecial and $(k+1)$-very ample, Proposition \ref{prop:lifting} says
\[ H^i(C^{[k+2]},S^iM_{k+2,B}\otimes N_{k+2,A}) = 
	H^i(C^{[i]},\wedge^{k+2}M_{i,A}\otimes S_{i,B}) \qquad \text{ for } 1\leq i \leq k+2.\]
For $i=1$, the problem becomes $H^1(C,\wedge^{k+1+i}M_{1,A}\otimes B) = 0$. If instead $i>1$, then we can use the fact that $\sO_{C^{[i]}}$ is a direct summand of $\sigma_{i-1,*}\sO_{C \times C^{[i-1]}}(\mathcal{Z}_{i-1})$ (see \eqref{eq:directsummand}) to reduce the problem to 
\[ H^i(C \times C^{[i-1]},\sigma_{i-1}^*(\wedge^{k+2}M_{i,A}\otimes S_{i,B}) \otimes \sO_{C^{[i-1]} \times C}(\mathcal{Z}_{i-1}))  = 0 \quad \text{ for } 2\leq i \leq k+2 \]
Using the short exact sequence \eqref{eq:pullbackMcurve}, together with Lemma \ref{lem:seqbundles}, we get a surjection 
\[ ( (B\otimes A^{-1}) \boxtimes (\wedge^{k+3}M_{i-1,A} \otimes S_{i-1,B}))(2\mathcal{Z}_{i-1}) \longrightarrow \sigma_{i-1}^*(\wedge^{k+2}M_{i,A} \otimes S_{i,B})(\mathcal{Z}_{i-1}) \]
of sheaves on $C^{[i-1]}\times C$.
Hence, it is enough to prove that the bundle on the left has no $H^i$. The Leray spectral sequence for the projection $\pr_C \colon C \times C^{[i-1]} \to C$ to the first factor shows that this vanishing is equivalent to
\[H^1(C,(B\otimes A^{-1})\otimes R^{i-1}\operatorname{pr}_{C,*}(\sO_C \boxtimes (\wedge^{k+3}M_{i-1,A} \otimes S_{i-1,B}))(2\mathcal{Z}_{i-1})) = 0,\]
and this is in particular true if the sheaf on $C$ is not supported on the whole curve. Thus it is enough to show that
\[ H^{i-1}(C^{[i-1]},\wedge^{k+3}M_{i-1,A}\otimes S_{i-1,B(2p_1)}) = 0 \]
for a general point $p_1\in C$. Proceeding in this way inductively, we see that our vanishings can be deduced from
\[ H^1(C,\wedge^{k+1+i}M_{A}\otimes B(2\xi_{i-1})) = 0  \quad \text{ for }  1\leq i \leq k+2\]
where $\xi_{i-1}=p_1+\cdots+p_{i-1} \in C^{[i-1]}$ is a general effective divisor of degree $i-1$. Observe now that the rank of $M_A$ is $r(A)$ so that $\wedge^{k+1+i}M_A^{\vee} \cong \wedge^{r(A)-k-1-i}M_A\otimes A$. Then the Serre duality proves that the previous vanishing is equivalent to
\[ H^0(C,\wedge^{r(A)-k-1-i}M_A \otimes \omega_C\otimes A\otimes B^{-1}\otimes (-2\xi_{i-1})) = 0 \quad \text{ for } 1\leq i \leq k+2. \]
By Green's vanishing theorem \cite[Theorem 3.a.1]{Green}, this is in particular implied by 
\[ h^0(C,\omega_C\otimes A\otimes B^{-1}(-2\xi_{i-1})) \leq r(A)-k-1-i \quad \text{ for all } 1\leq i \leq k+. \]
Since $\xi_{i-1}\in C^{[i-1]}$ is general, this follows from $h^1(C,B\otimes A^{-1}) = h^0(C,\omega_C\otimes A\otimes B^{-1}) \leq r(A)-k-2$. 
\end{proof}

\section{Determinantal ideals via commutative algebra}\label{sec:algebra}

\noindent This section is devoted to the proof of  Theorem \ref{thm:SS} from the introduction. Let $X$ be a smooth projective variety of dimension $n$, and $k$ be a non-negative integer. We assume that $n\leq 2$ or that $k\leq 1$. Let $A,B$ two line bundles on $X$, and set $L:=A\otimes B$. It is known from \cite[Theorem B]{CLPS25+} that the ideal $I(\Sigma_k(X,L))$ is generated in degree $k+2$ if $L$ is sufficiently ample. Furthermore, under the this assumption, it is also shown in \cite[Proposition 6.5]{CLPS25+} that 
$$
I(\Sigma_k(X,L))_{k+2} \cong H^0(X^{[k+2]},A_{k+2,L}).
$$
When $X=C$ is a curve of genus $g$, it is enough to assume that $\deg L \geq 2g+2k+2$ by \cite[Theorem 1.2]{ENP}. Recalling that $H^0(X^{[k+2]}, N_{k+2,A}) \cong \wedge^{k+2} H^0(X, A)$ and $H^0(X^{[k+2]}, N_{k+2,B}) \cong \wedge^{k+2} H^0(X, B)$, we may identify the map 
\[m_{A,B}^{k+2}\colon  \wedge^{k+2}H^0(X,A)\otimes \wedge^{k+2}H^0(X,B) \longrightarrow  I(\Sigma_k(X,L))_{k+2}\]
in  \eqref{eq:mABk+2} with the multiplication map 
\[  m^{k+2}_{A,B}\colon H^0(X^{[k+2]},N_{k+2,A}) \otimes H^0(X^{[k+2]},N_{k+2,B}) \longrightarrow H^0(X^{[k+2]},A_{k+2,L}) \]
of global sections on $X^{[k+2]}$. 
Now, note that the ideal $I(\Sigma_k(X,L))$ is generated by the $(k+2) \times (k+2)$-minors of $\operatorname{Cat}(A,B)$ if and only if the map $m^{k+2}_{A,B}$ is surjective. To show that the map $m^{k+2}_{A,B}$ is surjective, we take an algebraic approach in this section, and we employ  a cohomological approach in Section \ref{sec:cohomology}. 

\begin{lemma}\label{lem:mapsheaves}
Assume that $A,B$ are sufficiently ample line bundles on $X$. Then 
	The map $m^{k+2}_{A,B}$ is surjective if and only if the map 
	\begin{equation}\label{eq:maph} 
		h_*\sO_{X^{[k+2]}}(-\delta_{k+2}) \otimes h_*\sO_{X^{[k+2]}}(-\delta_{k+2}) \longrightarrow h_*\sO_{X^{[k+2]}}(-2\delta_{k+2})
	\end{equation}
	of sheaves on the symmetric product $X^{(k+2)}$ is surjective, where $h=h_{k+2} \colon X^{[k+2]} \to X^{(k+2)}$ is the Hilbert--Chow morphism.
\end{lemma}

\begin{proof}
	Via pushforward along the Hilbert--Chow morphism $h$, we can rewrite  $m^{k+2}_{A,B}$ as
	\[ m^{k+2}_{A,B}\colon H^0(X^{(k+2)},h_*N_{k+2,A}) \otimes H^0(X^{(k+2)},h_*N_{k+2,B}) \longrightarrow H^0(X^{(k+2)},h_*A_{k+2,A\otimes B}), \]
	and this map factors as the composition of the following two maps
	\begin{align*}
	\mu^{k+2}_{A,B} &\colon H^0(X^{(k+2)},h_*N_{k+2,A}) \otimes H^0(X^{(k+2)},h_*N_{k+2,B}) \longrightarrow H^0(X^{(k+2)},h_*N_{k+2,A}\otimes h_*N_{k+2,B}) \\
	\nu^{k+2}_{A,B} &\colon H^0(X^{(k+2)},h_*N_{k+2,A}\otimes h_*N_{k+2,B}) \longrightarrow H^0(X^{(k+2)},h_*A_{k+2,A\otimes B}).
	\end{align*}
	We have seen in \eqref{eq:isopushforward} that 
	\[ h_*N_{k+2,A} \cong \mathcal{N}_{k+2,A} \cong S_{k+2,A}\otimes \mathcal{N}_{k+2,\sO_X} \cong S_{k+2,A}\otimes h_*\sO_{X^{[k+2]}}(-\delta_{k+2})\] 
	with analogous isomorphisms for the line bundle $B$. Note that $S_{k+2, A}$ is sufficiently ample on $X^{(k+2)}$ as $A$ is. Then we see from Lemma \ref{lem:suffamplelemma} (4) that the multiplication map $\mu^{k+2}_{A,B}$ is surjective. Thus $m_{A,B}^{k+2}$ is surjective if and only if $\nu^{k+2}_{A,B}$ is surjective.  Since  $\nu^{k+2}_{A,B}$ is the map on global sections induced by the map of sheaves
	\begin{equation*} 
		(h_*\sO_{X^{[k+2]}}(-\delta_{k+2}) \otimes h_*\sO_{X^{[k+2]}}(-\delta_{k+2})) \otimes S_{k+2,L} \to 	
		h_*\sO_{X^{[k+2]}}(-2\delta_{k+2})\otimes S_{k+2,L}, 
	\end{equation*}
	it follows from Lemma \ref{lem:suffamplelemma} (1) that $\nu_{A,B}^{k+2}$ is surjective if and only if the map \eqref{eq:maph} is surjective. 
\end{proof}

This lemma yields a quick proof of Theorem \ref{thm:SS}.

\begin{proof}[Proof of Theorem  \ref{thm:SS}]
Recall that if $A,B$ are sufficiently ample on $X$ and $L:=A \otimes B$, then the ideal $I(\Sigma_k(X,L))$ is generated in degree $k+2$ by \cite[Theorem B]{CLPS25+}. Thus $I(\Sigma_k(X,L))$  is generated by the $(k+2) \times (k+2)$-minors of $\operatorname{Cat}(A,B)$ if and only if the map $m^{k+2}_{A,B}$ is surjective. By Lemma \ref{lem:mapsheaves}, $m^{k+2}_{A,B}$ is surjective if and only if the map  \eqref{eq:maph} of sheaves on $X^{(k+2)}$ is surjective. But the map \eqref{eq:maph} is the same as the surjective multiplication map $\mathcal{N}_{k+2,\sO_X}\otimes \mathcal{N}_{k+2,\sO_X} \to h_*\sO_{X^{[k+2]}}(-E_{k+2})$ in Proposition \ref{prop:idealXmsmooth}. 
\end{proof}

\begin{remark}
Following the reasoning of the proof, we see that Theorem \ref{thm:SS} is reduced to the statement that the map \eqref{eq:maph} of sheaves on the symmetric product is surjective. Since this is a local statement, it is independent of the particular smooth variety $X$. If it holds for one smooth variety $X$ of dimension $n$, then it holds for all such $X$. Since Raicu proved in \cite[Corollary 5.2]{Raicu} a stronger version of Theorem \ref{thm:SS} for $k=1$ and $X=\nP^n$, this leads to another proof of Theorem \ref{thm:SS} for $k=1$. 
\end{remark}

\section{Projective normality of Hilbert schemes in Grassmannians}\label{sec:projnorm}

\noindent Now, we want to use the ideas of Section \ref{sec:algebra} to prove Theorem \ref{thm:projnormalityhilb} from the introduction. We recall the setting. Let $X$ be a smooth variety of dimension $n\leq 2$, so either a curve or a surface, take $k\geq 0$, and let $L$ be a $(k+1)$-very ample line bundle on $X$. Then there is a well-defined map
\[ \varphi_{L,k+2}\colon X^{[k+2]} \longrightarrow G(k+2,H^0(X,L)) \]  
into the Grassmannian of $d$-dimensional quotients of $H^0(X,L)$: to each $\xi\in X^{[k+2]}$ one associates
\[ \varphi_{L,k+2}\colon \xi \longmapsto \left[H^0(X,L) \longtwoheadrightarrow H^0(X,L\otimes \sO_{\xi})\right] \]
This is the map induced by the globally generated tautological bundle $E_{k+2,L}$ and the composition with the Pl\"ucker embedding
\[  X^{[k+2]} \longrightarrow G(k+2,H^0(X,L)) \longhookrightarrow \nP(\wedge^{k+2}H^0(X,L)) \]
is precisely the map induced by the complete linear system of $N_{k+2,L} = \wedge^{k+2}E_{k+2,L}$. It was furthermore proved by Catanese--G\"ottsche \cite{CG} that the map $\varphi_{L, k+2}$ is a closed embedding if $L$ is $(k+2)$-very ample. 
We aim to show that if $L$ is sufficiently ample, then the embedding is projectively normal, meaning that all multiplication maps
\begin{equation}\label{eq:multmapell} 
	H^0(X^{[k+2]},N_{k+2,L})^{\otimes \ell} \longrightarrow H^0(X^{[k+2]},N_{k+2,L}^{\otimes \ell}) \qquad \text{ for all } \ell\geq 1 
\end{equation}
are surjective.

\begin{lemma}\label{lem:regularity}
	Let $X$ be a smooth projective variety of dimension $n\leq 2$ and let $L$ be a sufficiently ample line bundle on $X$. Then the embedding $\varphi_{L, k+2} \colon X^{[k+2]} \hookrightarrow \nP(\wedge^{k+2}H^0(X,L))$ is projectively normal if and only if the maps \eqref{eq:multmapell} are surjective for $1\leq \ell \leq n\cdot (k+2) + 1$.
\end{lemma}

\begin{proof}
	First we observe that if $L$ is sufficiently ample, then it is also $(k+2)$-very ample on $X$ thanks to Lemma \ref{lem:suffamplelemma}  so that the map $\varphi_{L, k+2}$ is an embedding, induced by the complete linear system of $N_{k+2,L}$. The Castelnuovo--Mumford regularity  $\reg(N_{k+2,L})$ is the smallest $r$ such that 
	\begin{equation}\label{eq:cmregularity} 
		H^i(X^{[k+2]},N_{k+2,L}^{\otimes (r-i)})=0 \qquad \text{ for all } i>0 .
	\end{equation}
	It is well known that the multiplication maps \eqref{eq:multmapell} are surjective for all $\ell > \reg(N_{k+2,L})$ so we are done if we can show the vanishings of \eqref{eq:cmregularity} for $r=(k+2)n+1$. Since in our hypotheses $X^{[k+2]}$ is smooth, an application of Kodaira vanishing theorem shows that it is enough to prove that $N_{k+2,L}\otimes  \omega_{X^{[k+2]}}^{-1}$ is ample on $X^{[k+2]}$. We have  $\omega_{X^{[k+2]}}\cong T_{k+2,\omega_X}\otimes \sO_{X^{[k+2]}}((n-2)\delta_{k+2})$, so that  $N_{k+2,L}\otimes  \omega_{X^{[k+2]}}^{-1}\cong T_{k+2,L\otimes \omega_X^{-1}}\otimes \sO((1-n)\delta_{k+2})$. If $n=1$ we can assume that $L\otimes \omega_X^{-1}$ is ample on $X$ so that $N_{k+2,L}\otimes  \omega_{X^{[k+2]}}^{-1} = T_{k+2,L\otimes \omega_X^{-1}}=S_{k+2,L\otimes \omega_X^{-1}}$ is ample on $X^{[k+2]}$. If instead $n=2$, then we can assume that $L\otimes \omega_X^{-1}$ is $(k+2)$-very ample on $X$ so that $N_{k+2,L}\otimes  \omega_{X^{[k+2]}}^{-1} = N_{k+2,L\otimes \omega_X^{-1}}$ is very ample on $X^{[k+2]}$.
\end{proof}

Now we can prove Theorem \ref{thm:projnormalityhilb}.

\begin{proof}[Proof of Theorem \ref{thm:projnormalityhilb}]
	Thanks to Lemma \ref{lem:regularity}, we need to prove that if $L$ is sufficiently ample on $X$, then all multiplication maps $H^0(X^{[k+2]},N_{k+2,L})^{\otimes \ell} \to H^0(X^{[k+2]},N_{k+2,L}^{\otimes \ell})$ are surjective for $1\leq \ell \leq n(k+2)+1$. By pushing forward to the symmetric product via the Hilbert--Chow morphism $h\colon X^{[k+2]} \to X^{(k+2)}$, and reasoning as in Lemma \ref{lem:mapsheaves} we are done if we can prove that the map of sheaves on $X^{(k+2)}$ 
	\[ \left( h_*N_{k+2,\sO_X} \right)^{\otimes \ell} \longrightarrow h_*(N_{k+2,\sO_X}^{\otimes \ell}) \]
	are surjective for all $1\leq \ell \leq n(k+2)+1$. If $n=1$, then this is clear since $h$ is an isomorphism. If $n=2$, then we see from \cite[Theorem 1.8]{ScalaDiag} that this map is the same as the multiplication map
	\[ \left( \mathcal{N}_{k+2,\sO_X} \right)^{\otimes \ell} \longrightarrow q^{\mathfrak{S}_{k+2}}_{k+2,*}(\mathcal{I}_{\Delta_{k+2}}^{\ell}\otimes \varepsilon_{k+2}^{\ell}),  \]
	where $\mathcal{I}_{\Delta_{k+2}}$ is the ideal sheaf of the big diagonal in the Cartesian product $X^{k+2}$. Working locally, we can assume that $X$ is an affine space, so that Proposition \ref{prop:idealXmsmooth} shows that the image of the map is given by $q_{k+2,*}^{\mathfrak{S}_{k+2}}(\mathcal{J}_{k+2}^{\ell}\otimes \varepsilon_{k+2}^{\ell})$, where $\mathcal{J}_{k+2}$ is the ideal sheaf on $X^{k+2}$ generated by alternating functions. Finally, since we are assuming $n=2$, Remark \ref{rmk:equalityideals} proves that $\mathcal{I}_{\Delta_{k+2}} = \mathcal{J}_{k+2}$ and we are done.
\end{proof}

\section{Determinantal ideals via cohomology}\label{sec:cohomology}

\noindent The aim of this section is to prove Theorem \ref{thm:EKS} and Theorem \ref{thm:SS} from the introduction using a cohomological method. This is essential in order to obtain the effective statement of Theorem \ref{thm:EKS}, but this points to the statement that one would like to prove in order to obtain an effective result for Theorem \ref{thm:SS}. Let again $X$ be a smooth projective variety of dimension $n$, and assume that $n\leq 2$ or $k\leq 1$. Choose two line bundles $A,B$ on $X$, and assume that $B$ is $(k+1)$-very ample, so that the tautological bundle $E_{k+2,B}$ is globally generated and we have an exact sequence of vector bundles:
\begin{align} 
 0 \longrightarrow M_{k+2,B} \longrightarrow &H^0(X,B)\otimes \sO_{X^{[k+2]}} \longrightarrow E_{k+2,B} \longrightarrow 0. \label{eq:seqMB} 
 \end{align} 
 If we now use Lemma \ref{lem:seqbundles} and we tensor by $N_{k+2,A}$, we obtain two exact sequences
 \begin{align}
  \cdots \longrightarrow  F_2 \longrightarrow  F_1 \longrightarrow  \wedge^{k+2} H^0(X,B) \otimes N_{k+2,A} \longrightarrow N_{k+2,A}\otimes N_{k+2,B}\longrightarrow  0 	\label{eq:exseqM}	\\
 \cdots \longrightarrow  G_2 \longrightarrow  G_1\longrightarrow  \wedge^{k+2} H^0(X,B) \otimes N_{k+2,A} \longrightarrow N_{k+2,A}\otimes N_{k+2,B}\longrightarrow  0 	\label{eq:dualexseq}
 \end{align}
 with 
 \[ F_i = \wedge^{k+2-i}H^0(X,B)\otimes S^iM_{k+2,B}\otimes N_{k+2,A}~~\text{ and }~~ G_i = \wedge^{k+2+i}H^0(X,B)\otimes S^iE_{k+2,B}^{\vee}\otimes N_{k+2,A}.
 \] 
 Then we can prove the following:
 
\begin{lemma}\label{lem:cohvanishing}
	Assume that either one of the following vanishings hold:
	\begin{align*} 
		&H^i(X^{[k+2]},S^iM_{k+2,B}\otimes N_{k+2,A}) = 0 \qquad \text{for } 1\leq i \leq k+2\\
		&H^i(X^{[k+2]},S^iE_{k+2,B}^{\vee}\otimes N_{k+2,A}) = 0 \qquad \text{ for } i>0.
	\end{align*}
	Then the multiplication map 
	$$
	m^{k+2}_{A,B}\colon H^0(X^{[k+2]},N_{k+2,A}) \otimes H^0(X^{[k+2]},N_{k+2,B}) \longrightarrow H^0(X^{[k+2]},A_{k+2,L})
	$$ 
	is surjective. 
\end{lemma}

\begin{proof}
	The map $m^{k+2}_{A,B}$ is obtained from the exact sequences \eqref{eq:exseqM} and \eqref{eq:dualexseq} by taking global sections in the right-most nonzero map. Hence, taking cohomology, we see that if $H^i(X^{[k+2]},F_i)=0$ for all $1\leq i \leq k+2$ or if $H^i(X^{[k+2]},G_i)=0$ for all $i>0$, then $m^{k+2}_{A,B}$ is surjective. These vanishing conditions are precisely those in our statement.
\end{proof}

\subsection{An effective result for the curve case} 
Let $C$ be a smooth curve of genus $g$, and $A, B$ be two line bundles on $C$. 
We can prove a more precise version of Theorem \ref{thm:EKS}. The statement should be compared with \cite[Theorem 1.3]{NP24+}

\begin{theorem}\label{thm:precisethmA}
	Assume that $A$ is nonspecial and $(k+1)$-very ample and
	\[ h^1(C,B\otimes A^{-1})\leq r(A)-k-2, \]
	where $r(A):=h^0(C,A)-1$. Then the map $m^{k+2}_{A,B}$ is surjective.
\end{theorem}

\begin{proof} 
	We first observe that under our hypotheses the line bundle $B$ is also nonspecial and $(k+1)$-very ample. This follows from \cite[Theorem 1.3]{NP24+}. However, the proof is short, so we repeat it here for completeness.  We need to show that $h^1(C,B(-\xi))=0$ for all $\xi\in C^{[k+2]}$. If this fails, then there is some $\xi \in C^{[k+2]}$ such that $H^0(X,\omega_C\otimes B^{-1}(\xi)) \ne 0$. Then \[
	h^1(C,B\otimes A^{-1}) = h^0(C,\omega_C\otimes A\otimes B^{-1}) \geq  h^0(C,A(-\xi)) = h^0(A)-k-2 > r(A)-k-2,
	\]
	which is a contradiction. 
	Now, thanks to Lemma \ref{lem:cohvanishing}, it is enough to check that 
	$$
	H^i(X^{[k+2]},S^iM_{k+2,B}\otimes N_{k+2,A}) = 0 \quad \text{ for } 1\leq i \leq k+2,
	$$
	but this is a consequence of Proposition \ref{prop:cohvanishingcurves}.
\end{proof} 

This yields numerical bounds on the degrees of $A,B$ such that the $(k+2) \times (k+2)$-minors of $\operatorname{Cat}(A,B)$ generate the ideal of the $k$-th secant variety $\Sigma_k(C, A \otimes B)$.

\begin{corollary}\label{cor:precisethmA}
	Assume that one of the following holds:
	\begin{enumerate}
		\item  $\deg(A)\geq \deg(B)\geq 2g+k+1$ and $\deg(A\otimes B)\geq 4g+2k+3$.
		\item  $\deg(A)=\deg(B)=2g+k+1$, and if $g>0$, then $A\ncong B$.
		\item  $A$ is general with $\deg(A)\geq g+2k+3$ and $\deg(B)=2g+k+1$.
	\end{enumerate}
	Then the ideal $I(\Sigma_k(C,A\otimes B))$ is generated by the $(k+2) \times (k+2)$-minors of $\operatorname{Cat}(A,B)$.
\end{corollary} 

\begin{proof}
	In any case, $\deg (A \otimes B) \geq 2g+2k+2$, so  the ideal $I(\Sigma_k(C,A\otimes B))$ is generated in degree $k+2$ by \cite[Theorem 1.2]{ENP}. Thus $I(\Sigma_k(C,A\otimes B))$ is generated by the $(k+2) \times (k+2)$-minors of $\operatorname{Cat}(A,B)$ if and only if $m^{k+2}_{A,B}$ is surjective. 	
	Now, note that both $A$ and $B$ are nonspecial and $(k+1)$-very ample in all cases. For the surjectivity of $m^{k+2}_{A,B}$, we need to check the condition of Theorem \ref{thm:precisethmA}, which can be rewritten as
	\[ \deg(B) \geq h^0(C,B\otimes A^{-1}) +2g+k+1\]
by Riemann--Roch. In the case (1), we  have $h^0(C,B\otimes A^{-1})\leq 1$ by degree reasons. If $\deg(B) \geq 2g+k+2$, then we are done. If instead $\deg(B)=2g+k+1$, then it must be that $\deg(A)>\deg(B)$, so that $h^0(C,B\otimes A^{-1})=0$ and we are done. In the cases (2) for $g>0$ and (3), it is straightforward to see that $h^0(C,B\otimes A^{-1})=0$, so we are done. In the case (2), if $g=0$, then $C^{[k+2]}=\nP^{k+2}$, so that the surjectivity of the multiplication map $m_{A,B}^{k+2}$ is clear. 
\end{proof}

Finally, we give the proof of Theorem \ref{thm:EKS}.

\begin{proof}[Proof of Theorem \ref{thm:EKS}]
The first statement follows from the cases (1) and (2) of Corollary \ref{cor:precisethmA}. For the second statement, we first show that if $\deg(L)\geq 4g+2k+2$, then the ideal $I(\Sigma_k(C,L))$ is determinantally presented. There are two line bundles $A,B$ on $C$ such that $L=A \otimes B$ and $\deg(A) \geq \deg(B)\geq 2g+k+1$. When $g>0$ and $\deg(A)=\deg(B)=2g+k+1$, we may assume that $A \ncong B$ possibly replacing $A,B$ with $A\otimes \eta,B\otimes \eta^{-1}$ for a general line bundle $\eta$ of degree zero on $C$. Then we are done by  the cases (1) and (2) of Corollary \ref{cor:precisethmA}. Next, we show that if $\deg(L) \geq 3g+3k+4$, then the ideal $I(\Sigma_k(C,L))$ is determinantally presented. To this end, set $A := \sO_C(p_1+\cdots+p_a)$ for $a:=\deg(L)-(2g+k+1)$ general points $p_1, \ldots, p_a\in C$ and $B:=L\otimes A^{-1}$ so that $\deg(B) = 2g+k+1$. Note that $A$ is nonspecial and $(k+1)$-very ample since $\deg(A)=a \geq g+2k+3$. 
We have $h^0(C,B\otimes A^{-1}) = h^0(C,L(-2p_1-\cdots-2p_a))=0$, so the condition of Theorem \ref{thm:precisethmA} is satisfied. Thus $m^{k+2}_{A,B}$ is surjective, so we are done as the ideal $I(\Sigma_k(C,A\otimes B))$ is generated in degree $k+2$ by \cite[Theorem 1.2]{ENP}.
\end{proof} 

\begin{remark}
By Theorem \ref{thm:EKS}, we see that the multiplication map  (\ref{eq:multmapell}) for $\ell=2$, that is $H^0(C^{[k+2]}, N_{k+2,L})^{\otimes 2} \to H^0(C^{[k+2]}, N_{k+2,L}^{\otimes 2})$, is surjective when $L$ is a line bundle on $C$ with $\deg L \geq 2g+k+2$. While we skip a detailed discussion, the argument presented here can be employed to establish that  the map (\ref{eq:multmapell}) is surjective for every $\ell \geq 1$. In other words, one can prove that if $\deg L \geq 2g+k+2$, then the embedding $C^{[k+2]} \subseteq \nP (H^0(C^{[k+2]}, N_{k+2, L}))$ is projectively normal. This is an effective result of Theorem \ref{thm:projnormalityhilb} for $n=1$. A sharper result is shown by Sheridan \cite{Sheridan} for the case $k=0$. 
\end{remark}

\subsection{Higher-dimensional case} 
We give the second proof of Theorem  \ref{thm:SS} using a cohomological method. Assume that $X$ is a smooth projective variety of dimension $n$ and $n\leq 2$ or $k\leq 1$. Here we only prove the second statement of Theorem  \ref{thm:SS} that if $L$ is sufficiently ample, then there is a splitting $L=A \otimes B$ with sufficiently ample line bundles $A,B$ on $X$ such that the ideal $I(\Sigma_k(X, L))$ is generated by the $(k+2) \times (k+2)$-minors of $\operatorname{Cat}(A,B)$. As the ideal $I(\Sigma_k(X, L))$ is generated in degree $k+2$ by \cite[Theorem B]{CLPS25+}, we only have to prove that the multiplication map $m^{k+2}_{A,B}$ is surjective. 
If $n\leq 2$ or  $k=0$, then Proposition \ref{prop:cohvanishing} and Lemma \ref{lem:cohvanishing} show that there are sufficiently ample line bundles $A,B$ on $X$ with $L=A \otimes B$ such that $m^{k+2}_{A,B}$ is surjective. Unfortunately, this approach does not work when $n\geq 3$ and $k=1$. What is missing in this case is the cohomology vanishing conditions of Lemma \ref{lem:cohvanishing}. We do expect them to hold also in this case, but we do not have a proof at the moment. However, it turns out that we can improve Lemma \ref{lem:cohvanishing} replacing the vanishings on $X^{[k+2]}$ by the same vanishings  on $X^{[k+1]}$.

\begin{lemma}\label{lem:cohvanishingbetter}
Assume that $n\leq 2$ or $k\leq 1$.
If $A,B$ are two sufficiently ample line bundles on $X$ such that 
\[ H^i(X^{[k+1]},S^iE^{\vee}_{k+1,B}\otimes N_{k+1,A})=0 \qquad \text{ for } i>0, \]
then the map $m^{k+2}_{A,B}$ is surjective.
\end{lemma}

\begin{proof}
Recall that $H^0(X, N_{k+2, B}) \cong \wedge^{k+2} H^0(X, B)$, and observe that $m^{k+2}_{A,B}$ is the global section map of 
\[\wedge^{k+2} H^0(X, B) \otimes N_{k+2,A} \longrightarrow N_{k+2,B} \otimes N_{k+2,A}. \]
Since $\sO_{X^{[k+2]}}$ is a direct summand of $\rho_*\sO_{X^{[k+1,k+2]}}(F_{k+1})$ by \eqref{eq:directsummand}, it is sufficient to prove that the map of sheaves
\[ \wedge^{k+2}H^0(X,B) \otimes \rho^*N_{k+2,A}(F_{k+1}) \longrightarrow (\rho^*N_{k+2,B}\otimes \rho^*N_{k+2,A})(F_{k+1}) \]
on $X^{[k+1,k+2]}$ is surjective on global sections. As $\rho^*N_{k+2,A}(F_{k+1}) \cong \res^*A \otimes \tau^*N_{k+1,A}$ by  \eqref{eq:pullbackNT}, this map is the same as
\[ \wedge^{k+2}H^0(X,B)\otimes \res^*A  \otimes \tau^*N_{k+1,A}\longrightarrow ( \res^*L \otimes \tau^*A_{k+1,L} )(-F_{k+1}). \]
Note that $(\res,\tau)_*\sO(-F_{k+1}) \cong \mathcal{I}_{\mathcal{Z}_{k+1}}$ since $(\res,\tau)$ is the blow-up of $X \times X^{[k+1]}$ along an irreducible variety $\mathcal{Z}_{k+1}$ with exceptional divisor $F_{k+1}$. 
Thus the previous map is surjective on global sections if and only if the map 
$$
\Phi\colon \wedge^{k+2}H^0(X,B)\otimes (A \boxtimes N_{k+1,A}) \longrightarrow (L \boxtimes A_{k+1,L}) \otimes \mathcal{I}_{\mathcal{Z}_{k+1}}
$$
is surjective on global section. The map $\Phi$ fits into a commutative diagram
$$
 \xymatrix{
  \wedge^{k+2}H^0(X,B)\otimes (A \boxtimes N_{k+1,A}) \ar@{=}[d] \ar[r]^-{\id_A \boxtimes \varphi} & A \boxtimes (M_{k+1,B} \otimes A_{k+1,L}) \ar[d]^-{\psi}\\
 \wedge^{k+2}H^0(X,B)\otimes (A \boxtimes N_{k+1,A}) \ar[r]_-{\Phi} & (L \boxtimes A_{k+1,L}) \otimes \mathcal{I}_{\mathcal{Z}_{k+1}}.
 }
$$
By Lemma \ref{lem:suffamplelemma} (2), we may assume that $B$ is $k$-very ample, so the kernel bundle $M_{k+1,B}$ is well-defined. Our task is then to show that the maps $\varphi$ and $\psi$ are surjective on global sections. 
Here, $\varphi$ is the last nonzero map of the exact sequence
$$
\cdots \longrightarrow F_2 \longrightarrow F_1 \longrightarrow \wedge^{k+2}H^0(X,B) \otimes N_{k+1,A}\overset{\varphi}{\longrightarrow} M_{k+1,B}\otimes A_{k+1,L} \longrightarrow 0
$$
induced from the short exact sequence
\[  0 \longrightarrow M_{k+1,B} \longrightarrow H^0(X,B) \otimes \sO_{X^{[k+1]}} \longrightarrow E_{k+1,B} \longrightarrow 0 \]
via Lemma \ref{lem:seqbundles}, where $F_i = \wedge^{k+2+i}H^0(X,B)\otimes S^iE_{k+1,B}^{\vee}\otimes N_{k+1,A}$ for $i\geq 1$. Thus $\varphi$ is surjective on global sections as soon as $H^i(X^{[k+1]},S^iE^{\vee}_{k+1,B}\otimes N_{k+1,A}) = 0$ for $i>0$, which is precisely our hypothesis. 
Next, if we denote by $\pr_{[X^{[k+1]}} \colon X \times X^{[k+1]} \to X^{[k+1]}$ the projection map, then 
$$
M_{k+1,B} \cong \pr_{X^{[k+1]},*} ((B \boxtimes \sO_{X^{[k+1]}}) \otimes \mathcal{I}_{\mathcal{Z}_{k+1}}).
$$ 
Thus the map $\psi$ is obtained from the map
$$
\sO_X \boxtimes M_{k+1,B} \longrightarrow (B \boxtimes \sO_{X^{[k+1]}}) \otimes \mathcal{I}_{\mathcal{Z}_{k+1}} 
$$
after tensoring with $A \boxtimes A_{k+1,L}$. Note that the map induced by $\psi$ on global sections can be identified with the multiplication map
	\[ H^0(X \times X^{[k+1]}, (B \boxtimes A_{k+1,L})\otimes \mathcal{I}_{\mathcal{Z}_{k+1}} ) \otimes H^0(X,A) \longrightarrow H^0(X \times X^{[k+1]},(L \boxtimes A_{k+1,L}) \otimes \mathcal{I}_{\mathcal{Z}_{k+1}}). \]
If we set $\mathcal{F}:=(\id_X \times h_{k+1})_*((\sO_X \boxtimes \sO_{X^{[k+1]}}(-2\delta_{k+1}))\otimes \mathcal{I}_{\mathcal{Z}_{k+1}})$, then this map is the same as the multiplication map
$$
	H^0(X \times X^{(k+1)}, \mathcal{F}\otimes (B \boxtimes S_{k+1,L})) \otimes H^0(X,A) \longrightarrow H^0(X \times X^{(k+1)},\mathcal{F}\otimes (L \boxtimes S_{k+1,L})).
$$
Since $S_{k+1,L}$ is sufficiently ample on $X^{(k+1)}$, this map is surjective by  Lemma \ref{lem:suffamplelemma} (5). Hence $\psi$ is surjective on global sections.
\end{proof}

Now, we can give another proof of Theorem \ref{thm:SS}.

\begin{proof}[Second Proof of Theorem \ref{thm:SS}]
As mentioned before, we only prove the second statement. Since $L$ is sufficiently ample, we may find sufficiently ample line bundles $A,B$ on $X$ with $L=A \otimes B$ such that $H^i(X^{[k+1]},S^iE_{k+1,B}^{\vee}\otimes N_{k+1,A})=0$ for $i>0$ thanks to Proposition \ref{prop:cohvanishing}. Then Lemma \ref{lem:cohvanishingbetter} shows that the map $m^{k+2}_{A,B}$ is surjective, and hence, the ideal $I(\Sigma_k(X,L))$ is determinantally presented.
\end{proof}

\section{Rank three quadratic equations  for projective schemes}\label{sec:quadrics43}

\noindent We now turn to the case of an arbitrary projective scheme $X$. Our aim is to prove Theorem \ref{thm:HLMP} from the introduction. Along the way, we recover main results of \cite{HLMP2021} and \cite{SS2011}. The key will be given by the surjectivity of the two multiplication maps
\begin{align*}
m^2_{A,B}&\colon \wedge^2 H^0(X,A) \otimes \wedge^2 H^0(X,B)  \longrightarrow I(X,L)_2 \\
s^2_{A,B}&\colon S^2H^0(X,A) \otimes \wedge^2 H^0(X,B) \longrightarrow \wedge^2 H^0(X,L)
\end{align*}
where $A,B$ are sufficiently ample line bundles on $X$ and $L:=A \otimes B$. As $L$ is also sufficiently ample, we may assume that $L$ is very ample by Lemma \ref{lem:suffamplelemma} (2).
The following technical result is inspired by Lemma \ref{lem:cohvanishingbetter}, and indeed we could prove it via Hilbert schemes of points if $X$ is smooth. However, since now $X$ is an arbitrary projective scheme, the Hilbert scheme of points on $X$ may be arbitrarily bad. Therefore, we take a more direct approach.

\begin{lemma}\label{lem:cohvanishingexplicit}
Let $X$ be a projective scheme, $A,B$ be two globally generated line bundles on $X$, and $L:=A\otimes B$. Assume that the following hold:
	\begin{enumerate} 
	\item $H^1(X,A) =H^1(X, L)=0$.
	\item $H^i(X,M_{A}\otimes B) = H^i(X,M_{B}\otimes A) = 0$ for $i=1,2$.
	\item $H^1(X,\wedge^2 M_B \otimes A) = 0 $.
	\item  $H^1(X,M_A^{\otimes 2}\otimes B) = H^1(X,M_A\otimes M_B\otimes L)=0$.
	\end{enumerate}
	  Then both maps $m^2_{A,B}$ and $s^2_{A,B}$ are surjective.
\end{lemma}
\begin{proof}
	Consider first the multiplication map
	$$
	\Phi\colon \wedge^2 H^0(X,B) \otimes H^0(X,A)^{\otimes 2} \longrightarrow H^0(X,M_L\otimes L) \subseteq H^0(X,L)^{\otimes 2}.
	$$
	As $\operatorname{char}(\kk)\ne 2$, we have $H^0(X,A)^{\otimes 2} = \wedge^2 H^0(X,A) \oplus S^2H^0(X,A)$ and $H^0(X,M_L\otimes L) = I(X,L)_2 \oplus \wedge^2 H^0(X,L)$. Furthermore, the map $\Phi$ respects this decomposition. Hence we only need to prove that $\Phi$ is surjective onto $H^0(X,M_L\otimes L)$. We may factor $\Phi$ as
	\[ \wedge^2 H^0(X,B) \otimes H^0(X,A)^{\otimes 2} \xrightarrow{\varphi\otimes \id_{H^0(A)}} H^0(X,M_B\otimes L)  \otimes H^0(X,A) \xrightarrow{~\psi~}H^0(X,M_L\otimes L).  \]
 The condition (3) shows that $\varphi\otimes \id_{H^0(A)}$ is surjective. It suffices to check that $\psi$ is surjective. Since $H^0(X,M_B\otimes L) = H^0(X,M_L\otimes B)$, the surjectivity of $\psi$ follows from $H^1(X,M_A\otimes M_L\otimes B)=0$.  
By the condition (1) and Lemma \ref{lem:M_L}, it is enough to show that $H^1(X^2,((M_A\otimes B)\boxtimes L )\otimes \mathcal{I}_{\Delta_{1,1}})=0$. Consider the short exact sequence 
	\[
	0 \longrightarrow (M_A \otimes B) \boxtimes (M_B \otimes A)  \longrightarrow (M_A \otimes B) \boxtimes (A \otimes H^0(B))  \longrightarrow  ( (M_A \otimes B) \boxtimes L)  \longrightarrow  0.
	\]	
After tensoring by $\mathcal{I}_{\Delta_{1,1}}$, we see that the desired cohomology vanishing follows from
	\[
	H^1(X^2, ((M_A \otimes B) \boxtimes A) \otimes \mathcal{I}_{\Delta_{1,1}})=0~~\text{ and }~~H^2(X^2, ((M_A \otimes B) \boxtimes (M_B \otimes A)) \otimes \mathcal{I}_{\Delta_{1,1}})=0.
	\]
By Lemma \ref{lem:M_L}, the first vanishing  holds by the conditions (1) and (4). Consider the short exact sequence $0 \to \mathcal{I}_{\Delta_{1,1}} \to \sO_{X^2} \to \sO_X \to 0$. The second vanishing can be deduced from 
$$
H^1(X, M_A \otimes M_B \otimes L) = 0~~\text{ and }~~H^2(X^2, (M_A \otimes B) \boxtimes (M_B \otimes A))=0,
$$
which hold by  the conditions (2) and (4). 
\end{proof}

We can now give alternative proofs of \cite[Theorem 1.1 and Theorem 1.3]{SS2011}.

\begin{corollary}\label{cor:cohvanishingexplicit}
$(1)$ Let $X$ be a projective scheme, and $A,B$ be sufficiently ample line bundles on $X$.Then both maps $m^2_{A,B}$ and $s^2_{A,B}$ are surjective. In particular, the ideal $I(X,L)$ is generated by the $2\times 2$ minors of $\operatorname{Cat}(A,B)$.\\[3pt]
$(2)$ Let $X$ be a smooth projective variety of dimension $n$, and $A := \omega_X\otimes H^{\otimes j_1}\otimes M_1, B := \omega_X\otimes H^{\otimes j_2}\otimes M_2$, where $H$ is a very ample line bundle and $M_1,M_2$ are nef line bundles on $X$. If $j_1,j_2\geq n+2$ or $j_1,j_2\geq n+1$ and $(X,H)\ne (\nP^n,\sO_{\nP^n}(1))$, then both maps $m^2_{A,B}$ and $s^2_{A,B}$ are surjective. In particular, the ideal $I(X,L)$ is generated by the $2\times 2$ minors of $\operatorname{Cat}(A,B)$, where $L:=A \otimes B$.
\end{corollary}

\begin{proof}
$(1)$ As $A,B$ are sufficiently ample, the conditions of Lemma \ref{lem:cohvanishingexplicit} holds by Fujita--Serre vanishing and Lemma  \ref{lem:suffamplelemma} (7).  This proves that the maps $m^2_{A,B},s^2_{A,B}$ are surjective. To conclude, it is enough to show that the ideal $I(X,L)$ is generated in degree two, but this is true if $L=A\otimes B$ is sufficiently ample because of Lemma \ref{lem:suffamplelemma}(8).

\medskip

\noindent $(2)$ We check the conditions of Lemma \ref{lem:cohvanishingexplicit}. The condition (1) follows from Kodaira vanishing, while all other vanishings apart from $H^1(X,M_A\otimes M_B\otimes L)=0$ follow from \cite[Theorem 2.1 and Proposition 3.1]{EL1993}. The remaining one can be proven as in \cite{EL1993}.\footnote{Alternatively, one can prove $H^1(X,M_A\otimes M_B\otimes L)=0$ as follows. By Lemma \ref{lem:M_L}, it is enough to show that 
$H^1(X^3, (A \boxtimes B \boxtimes L) \otimes \mathcal{I}_{\Delta_{1,3}} \otimes \mathcal{I}_{\Delta_{2,3}})=0$. By \cite[Lemma 1.2]{ELY}, $ \mathcal{I}_{\Delta_{1,3}} \otimes \mathcal{I}_{\Delta_{2,3}} =  \mathcal{I}_{\Delta_{1,3}} \cdot \mathcal{I}_{\Delta_{2,3}}=\mathcal{I}_{\Delta_{1,3} \cup \Delta_{2,3}}$. Let $b \colon Y \to X^3$ be the blow-up along  $\Delta_{1,3} \cup \Delta_{2,3}$ with exceptional divisor $E$. We need to check that
$H^1(Y, b^*(A \boxtimes B \boxtimes L)(-E))=0$. Note that $Y$ is smooth and $b^*(H \boxtimes H \boxtimes H^{\otimes 2})(-E)$ is globally generated. As $\omega_Y = b^*(\omega_X^{\boxtimes 3})(-(n-1)E)$, the desired cohomology vanishing follows from Kawamata--Viehweg vanishing theorem.}
\end{proof}

From now on, we focus on rank three quadratic equations for projective schemes. 
The following is a special case of \cite[Theorem 1.1]{HLMP2021}, but we give a quick proof using representation theory, which works in positive characteristic different from 2 and 3. 

\begin{proposition}\label{prop:rank3veronese}
The ideal  $I(\nP^n,\sO_{\nP^n}(2))$ is generated by quadrics of rank $3$. 
\end{proposition}

\begin{proof}
It is well known that the  ideal $I(\nP^n,\sO_{\nP^n}(2))$ is generated by quadrics. It is enough to show that $I(\nP^n,\sO_{\nP^n(2)})_2$ is spanned by quadrics of rank $3$. We will use the language of representation theory. Let $\nP^n = \nP(V)$ for a vector space $V$ of dimension $n+1$. Then there is an exact sequence
	\[0 \longrightarrow  I(\nP(V),\sO_{\nP(V)}(2))_2 \longrightarrow S^2(S^2V) \longrightarrow S^4V \longrightarrow 0  \] 
	of $\operatorname{GL}(V)$-representations. We have an irreducible representation decomposition $S^2(S^2V) \cong S^{(2,2)}V \oplus S^{4}V$ as representations of $\operatorname{GL}(V)$. Thus $I(\nP(V),\sO_{\nP(V)}(2))_2 \cong S^{(2,2)}V$.  Now, we recall the classical relation between quadrics of rank $3$ and  two-dimensional subspace $W\subseteq V$. If $\sigma_0,\sigma_1$ is a basis of $W$, define $x_0:=\sigma_0^2,x_1 := \sigma_1\sigma_2,x_2
:=\sigma_2^2 \in H^0(\nP(V),\sO_{\nP(V)}(2))$. It is straightforward to see that
	\[ q := \det\begin{pmatrix} \sigma_0^2 & \sigma_{0}\sigma_1 \\ \sigma_0\sigma_1 & \sigma_1^2 \end{pmatrix} = x_0x_2-x_1^2 \]
	is a quadric of rank exactly three in $I(\nP(V),\sO_{\nP(V)}(2))_2$. Up to a scalar, this quadric is independent of the basis $\sigma_0,\sigma_1$, so that this defines a map
	\[ q\colon \operatorname{Gr}(2,V) \longrightarrow \nP(I(\nP(V),\sO_{\nP(V)}(2))_2^{\vee}) \] 
	whose image is contained in the space of quadrics of rank at most three. All constructions are $\operatorname{GL}(V)$-equivariant, so the image of this map must span a nonempty $\operatorname{GL}(V)$-invariant subspace of the image. Since $I(\nP(V),\sO_{\nP(V)}(2))_2$ is an irreducible representation of $\operatorname{GL}(V)$, the only such subspace is the whole space.
\end{proof} 

\begin{remark}
	It turns out that  the map $q \colon \operatorname{Gr}(2,V) \rightarrow \nP(I(\nP(V),\sO_{\nP(V)}(2))_2^{\vee})$ in the proof of Proposition \ref{prop:rank3veronese} is the embedding induced by the complete linear system $H^0(\operatorname{Gr}(2,V),\sO_{\operatorname{Gr}(2,V)}(2))$. 
	Indeed, one can see explicitly that the map $q$ is defined by quadrics in the Pl\"ucker coordinates of $\operatorname{Gr}(2,V)$, and since we showed in the proof of Proposition \ref{prop:rank3veronese} that the image is not contained in any hyperplane, it must be the map induced by a sub-linear system of $H^0(\operatorname{Gr}(2,V),\sO_{\operatorname{Gr}(2,V)}(2))$. But this is actually the complete linear system since $H^0(\operatorname{Gr}(2,V),\sO_{\operatorname{Gr}(2,V)}(2)) \cong S^{(2,2)}V \cong I(\nP(V),\sO_{\nP(V)}(2))_2$. As $\sO_{\operatorname{Gr}(2,V)}(2)$ is very ample, we conclude that the map $q$ must be an embedding. This was also proven in \cite[Corollary 1.4]{Park24} by a different method. 
\end{remark}

As a consequence of Proposition \ref{prop:rank3veronese}, we get the following (cf. \cite[Theorem 1.3]{HLMP2021}).

\begin{corollary}\label{cor:squarerank3}
Let $X$ be a projective scheme, and $H$ be a very ample line bundle on $X$ such that $X \subseteq \nP (H^0(X, H))$ is arithmetically normal and the ideal $I(X, H)$ is generated by quadrics. Then the ideal $I(X, H^{\otimes 2})$ is generated by quadrics of rank $3$. 
\end{corollary}

\begin{proof}
We follow the arguments of \cite[Proof of Theorem 1.3]{HLMP2021}. Set $\nP^r := \nP (H^0(X, H))$.
Note that $X \subseteq \nP (H^0(X, H^{\otimes 2}))$ is a linear section of $v_2(\nP^r) \subseteq \nP (H^0(\nP^r, \sO_{\nP^r}(2)))$. Since the ideal $I(\nP^r, \sO_{\nP^r}(2))$ is generated by quadrics of rank $3$ by Proposition \ref{prop:rank3veronese}, it follows that the ideal $I(X, H^{\otimes 2})$ is also generated by quadrics of rank $3$. 
\end{proof}

\begin{remark}
We give an elementary proof of Corollary \ref{cor:squarerank3} under a slightly different assumption. Suppose that $A=B=H$ is a very ample line bundle on a projective scheme $X$ satisfying the conditions of Lemma \ref{lem:cohvanishingexplicit} so that $m_{H,H}^2$ is surjective. These conditions also imply that $X \subseteq \nP (H^0(X, H))$ is arithmetically normal and the ideal $I(X, H)$ is generated by quadrics. We now claim that  the ideal $I(X, H^{\otimes 2})$ is generated by quadrics of rank $3$. We denote by $m_{ij}$ the $(i,j)$-entry of $\operatorname{Cat}(H,H)$ (note that $\operatorname{Cat}(H,H)$ is symmetric so that $m_{ij}=m_{ji}$). As $I(X, H^{\otimes 2})$ is generated by $m_{ij}m_{k\ell} - m_{i \ell} m_{kj}$, it is enough to show that $m_{ij}m_{k\ell} - m_{i \ell} m_{kj}$ is a linear combination of quadrics of rank $3$ in  $I(X, H^{\otimes 2})$.
We have three cases: (1) $i=j$ and $k=\ell$, (2) $i=j$ and $k \neq \ell$ (or $i \neq j$ and $k=\ell$), (3) $i \neq j$ and $k \neq \ell$. In the case (1), we have
$$
m_{ii}m_{kk}-m_{ik}m_{ki} = m_{ii}m_{kk} - m_{ik}^2,
$$
which is already a quadric of rank $3$. In the case (2), we have\\[-20pt]

\begin{footnotesize}
$$
m_{ii}m_{k\ell} - m_{i\ell}m_{ki} =\frac{1}{2} \left[ (m_{ii}m_{kk} - m_{ik}^2) +(m_{ii}m_{\ell \ell}- m_{i\ell}^2) + \big( (m_{ik} - m_{i\ell})^2+ m_{ii}(2m_{k\ell} - m_{kk} - m_{\ell \ell}) \big) \right],
$$
\end{footnotesize}

\noindent which is a linear combination of quadrics of rank $3$ in $I(X, H^{\otimes 2})$. In the case (3), one can also explicitly write $m_{ij}m_{k\ell} - m_{i \ell} m_{kj}$ as a linear combination of quadrics of rank $3$ in $I(X, H^{\otimes 2})$. An explicit expression can be found in \cite[Proof of Theorem 3.1]{HLMP2021}. 
\end{remark}	

\begin{lemma}\label{lem:quadricsrank}
Let $X$ be a projective scheme, $A$ be a globally generated line bundle on $X$, and $B$ be a very ample line bundle on $X$ such that $I(X,B)_2$ is spanned by quadrics of rank at most $r \geq 1$. Set $L:=A \otimes B$.  If  the multiplication map
	\[ S^2H^0(X,A) \otimes I(X,B)_2 \longrightarrow I(X,L)_2\]
	is surjective, then  $I(X,L)_2$ is also spanned by quadrics of rank at most $r$.   
\end{lemma} 

\begin{proof}
A rank $r$ quadric $q$ in $I(X, A)_2$ can be written as $q = \beta_1^2+\cdots+\beta_r^2$ for some  $\beta_i\in H^0(X,B)$. For $\alpha^2 \in S^2 H^0(X, A)$ with $\alpha \in H^0(X, A)$, the image of $\alpha^2 \otimes q$ under the map in the lemma is a rank $r$ quadric $(\alpha \cdot \beta_1)^2 + \cdots + (\alpha \cdot \beta_r)^2$. Since $S^2 H^0(X, A)$ is spanned by $\alpha^2$ and $I(X, L)_2$ is spanned by the image of $\alpha^2 \otimes q$, the lemma follows. 
\end{proof}

Now, we are ready to prove Theorem \ref{thm:HLMP}.
	
\begin{proof}[Proof of Theorem \ref{thm:HLMP}]
Let $X$ be a projective scheme, and $A,B$ be two line bundles on $X$. We have a commutative diagram
	\[
	\xymatrixcolsep{0.8in}
	\xymatrix{
		S^2H^0(X,A) \otimes \wedge^2 H^0(X,B) \otimes \wedge^2 H^0(X,B) \ar[r]^-{s^{2}_{A,B} \otimes \id_{\wedge^2 H^0(B)}} \ar[d]_-{\id_{S^2H^0(A)} \otimes m^2_{B,B}} & \wedge^2H^0(X,A\otimes B) \otimes \wedge^2 H^0(X,B) \ar[d]^-{m^2_{A\otimes B,B}} \\
		S^2H^0(X,A)\otimes I(X,B^{\otimes 2})_2  \ar[r] & I(X,A\otimes B^{\otimes 2})_2.
	}
	\]
	Assume that the following conditions hold:
	\begin{itemize}
		\item[$(a)$]  $s^2_{A,B}$ and $m^2_{A\otimes B,B}$ are surjective. 
		\item[$(b)$]  $X \subseteq \nP(H^0(X,B))$ is arithmetically normal, and $I(X, B)$ is generated by quadrics.
		\item[$(c)$] $I(X, A\otimes B^{\otimes 2})$ is generated by quadrics.
	\end{itemize}
	The condition $(b)$ and Corollary \ref{cor:squarerank3} show that the ideal $I(X,B^{\otimes 2})$ is generated by quadrics of rank $3$. Then the condition $(a)$  and Lemma \ref{lem:quadricsrank} imply that the space $I(X,A\otimes B^{\otimes 2})_2$ is spanned by quadrics of rank $3$, and thanks to the condition $(c)$, these quadrics generate the whole ideal $I(X,A\otimes B^{\otimes 2})$. Hence we only need to verify the conditions $(a), (b), (c)$. 
	
\medskip	
	
\noindent (First Part) If $A,B$ are sufficiently ample, then the condition $(a)$ is guaranteed by Corollary \ref{cor:cohvanishingexplicit}, and conditions $(b)$ and $(c)$ are guaranteed by  Lemma \ref{lem:suffamplelemma} (8). 
	
\medskip

\noindent (Second Part) Assume that $X$ is a smooth projective variety of dimension $n$, fix a very ample line bundle $H$ and a nef line bundle $M$ on $X$, and set $A:= \omega_X\otimes H^{\otimes j_1}\otimes M$ and $B:=\omega_X \otimes H^{\otimes j_2}$ for $j_1=j_2=n+2$. Then the condition $(a)$ holds because of Corollary \ref{cor:cohvanishingexplicit} and conditions $(b)$ and $(c)$ hold because of \cite[Theorem 2.1]{EL1993}.  The previous reasoning shows that the ideal $I(X,L)$ is generated by quadrics of rank $3$, where $L:=A\otimes B^{\otimes 2}$.
\end{proof}

Finally, we recover \cite[Theorem 1.1]{HLMP2021}. 

\begin{corollary}\label{cor:idealVeronese}
The ideal $I(\nP^n, \sO_{\nP^n}(d))$ is generated by quadrics of rank 3 for $d \geq 2$. 
\end{corollary}

\begin{proof}
It is immediate from Proposition \ref{prop:rank3veronese} for $d=2$ and Theorem \ref{thm:HLMP} for $d \geq 3$.
\end{proof}

\begin{remark}\label{rem:positivechar}
We argue that Theorem \ref{thm:HLMP} holds for projective schemes $X$ over an algebraically closed field $\kk$ with $\operatorname{char}(\kk)\ne 2$. This fully confirms the Han--Lee--Moon--Park conjecture \cite[Conjecture 6.1]{HLMP2021}. Indeed, all results and arguments in this section are valid as long as $\operatorname{char}(\kk) \neq 2,3$ except for Corollary \ref{cor:cohvanishingexplicit} (2) and the second part of Theorem \ref{thm:HLMP}. If $\operatorname{char}(\kk)  = 3$, then \cite[Theorem 1.2]{HLMP2021} shows that Proposition \ref{prop:rank3veronese} does not hold  (in this case, $I(\nP(V), \sO_{\nP(V)}(2))_2$ is no longer an irreducible representation of $\operatorname{GL}(V)$, as was explained to us by Claudiu Raicu). However, it was recently shown in \cite{VerChar3}  that the ideal $I(\nP^n, \sO_{\nP^n}(3))$ is generated by quadrics of rank 3 in characteristic 3. As a consequence, $I(X, H^{\otimes 3})$ is generated by quadrics of rank 3 whenever $H$ is a sufficiently ample line bundle on $X$. As in the proof of Theorem \ref{thm:HLMP}, one can verify that the map
$$
S^2 H^0(X, A) \otimes I(X, B^{\otimes 3})_2 \longrightarrow I(X, A \otimes B^{\otimes 3})_2
$$
is surjective for sufficiently ample line bundles $A,B$ on $X$. This implies that  the first statement of Theorem \ref{thm:HLMP} holds in characteristic 3 as well. Furthermore, if $X = \nP^n$ and  $\operatorname{char}(\kk) \neq 2$, then Corollary \ref{cor:cohvanishingexplicit} (2) and Corollary \ref{cor:idealVeronese} ($d \geq 3$ in characteristic 3) are also true as one can verify all required cohomology vanishing on $\nP^n$.
\end{remark}

\bibliographystyle{ams}

\begin{thebibliography}{[1]}
	
\bibitem{Ago22}
Daniele Agostini, \textit{Asymptotic syzygies and higher order embeddings}, Int. Math. Res. Not., \textbf{2022} (2022), 293--2967

\bibitem{Ago24}
Daniele Agostini, \textit{The Martens--Mumford Theorem and the Green--Lazarsfeld secant conjecture}, J. Algebraic Geom. \textbf{33} (2024), 629--654. 

\bibitem{OscarTutorial}
Daniele Agostini, Jinhyung Park \textit{A computation with multisymmetric polynomials}, OSCAR Tutorial (2025), \url{https://www.oscar-system.org/tutorials} 


\bibitem{ABW82}
Kaan Akin, David A. Buchsbaum, and Jerzy Weyman, \textit{Schur functors and Schur complexes}, Adv. Math. \textbf{44} (1982), 207--278. 

\bibitem{BBF24+} 
 Weronika Buczy\'nska, Jaros{\l}aw Buczy\'nski, and {\L}ucja Farnik,
\textit{Cactus varieties of sufficiently ample embeddings of projective schemes have determinantal equations}, preprint (2024), arXiv:2412.00709.

\bibitem{BGL13}
Jaros{\l}aw Buczy\'nski, Adam Ginensky, and J. M. Landsberg, \textit{Determinantal equations for secant varieties and the Eisenbud--Koh--Stillman conjecture}, J. London Math. Soc. \textbf{88} (2013), 1--24.


\bibitem{CJ01}
Michael L. Catalano-Johnson, \textit{The homogeneous ideals of higher secant varieties}, J. Pure Applied Algebra \textbf{24} (2001), 123--129.

\bibitem{CG}
Fabrizio Catanese and Lothar G\"ottsche \textit{$d$-very-ample line bundles and embeddings of Hilbert schemes of $0$-cycles}, Manuscripta Math. \textbf{68} (1990), 337--342. 

\bibitem{Choe25+}
Junho Choe, \textit{Quadratic equations of tangent varieties via four-way tensors of linear forms}, preprint (2025), arXiv:2510.01908.

\bibitem{CKP23+}
Junho Choe, Sijong Kwak, and Jinhyung Park, \textit{Syzygies of secant varieties of smooth
projective curves and gonality sequences}, preprint (2023), arXiv:2307.03405.

\bibitem{CLPS25+}
Doyoung Choi, Justin Lacini, Jinhyung Park, and John Sheridan, \textit{ Singularities and syzygies of secant varieties of smooth projective varieties}, preprint (2025), arXiv:2502.19703.

\bibitem{EL1993}
Lawrence Ein and Robert Lazarsfeld, \textit{Syzygies and Koszul cohomology of smooth projective varieties of arbitrary dimension}, Invent. Math. \textbf{111} (1993), 51--67.

\bibitem{EL15}
Lawrence Ein and Robert Lazarsfeld, \textit{The gonality conjecture on syzygies of algebraic curves of large degree} Publ. Math., Inst. Hautes \'{E}tud. Sci.  \textbf{122} (2015) 301--313.

\bibitem{ELY}
Lawrence Ein and Robert Lazarsfeld, and David Yang, \textit{A vanishing theorem
	for weight-one syzygies}, Algebra Number Theory \textbf{10} (2016), 1965--1981.

\bibitem{Eis2005}
David Eisenbud, \textit{The geometry of syzygies}, Graduate Texts in Math. \textbf{229} (2005), Springer-Verlag, Berlin.

\bibitem{EKS1988}
David Eisenbud, Jee Koh, and Michael Stillman, \textit{Determinantal equations for curves of high degree}, Amer. J. Math. \textbf{110} (1988), 513--539.

\bibitem{ENP}
	Lawrence Ein, Wenbo Niu, and Jinhyung Park,
	\newblock \textit{Singularities and syzygies of secant varieties of nonsingular
	projective curves},
	\newblock Invent. Math. \textbf{222} (2020), 615--665.
	
\bibitem{ESGood}
 Torsten Ekedahl and Roy Skjelnes, \textit{ 
 Recovering the good component of the Hilbert scheme} Ann. Math. \textbf{179} (2014), 805--841.
 
 \bibitem{Fisher}
 Tom Fisher, \textit{Pfaffian representations of elliptic normal curves}, Trans. Amer. Math. Soc. \textbf{362} (2010), 2525--2540.
  
  \bibitem{M2}
Daniel R. Grayson and  Michael E. Stillman,
\textit{Macaulay2, a software system for research in algebraic geometry}, available at \url{http://www2.macaulay2.com}.

   \bibitem{Gin2010}
  Adam Ginensky, \textit{A generalization of the Clifford index and determinantal equations for curves and their secant varieties}, preprint (2010), arXiv:1002.2023.
  
  \bibitem{Green}
   Mark L. Green,  \textit{Koszul cohomology and the geometry of projective varieties}, J. Differential Geom. \textbf{19}, 125--171, (1984) 
  
  \bibitem{Haiman}
  Mark Haiman, \textit{Hilbert schemes, polygraphs and the {M}acdonald positivity conjecture},   J. Amer. Math. Soc. \textbf{14} (2001), 941--1006. 
  
  \bibitem{Haimanlectures}
   Mark Haiman, \textit{Commutative algebra of n points in the plane} (with an appendix by Ezra Miller),
  Trends in Commutative Algebra, MSRI Publications \textbf{51} (2004), 153--180. 

   \bibitem{HLMP2021}
   Kangjin Han, Wanseok Lee, Hyungsuk Moon, and Euisung Park, \textit{Rank 3 quadratic generators of Veronese embeddings}, Compositio Math. \textbf{157} (2021), 2001--2025.
   
    \bibitem{HTT}
 Yukitoshi Hinohara, Kazuyoshi Takahashi, and Hiroyuki Terakawa, \textit{On tensor products of $k$-very ample line bundles}, Proc. Amer. Math. Soc. \textbf{133} (2005), 687--692.

   
   \bibitem{Inamdar}
   S.P. Inamdar, \textit{On syzygies of projective varieties}, Pac. J.  Math. \textbf{177} (1997), 71--76.
   
 \bibitem{LazPosI}
 Robert Lazarsfeld, \textit{Positivity in Algebraic Geometry I $\&$ II}, A Series of Modern Surveys in Math. \textbf{48} $\&$ \textbf{49} (2004), Springer-Verlag, Berlin.
 
 
 \bibitem{VerChar3} 
 Donghyeop Lee, Euisung Park, and Saerom Sim, \textit{Rank 3 quadratic generators of Veronese embeddings: the characteristic 3 case}, preprint (2025) arXiv:2509.09383.
 
 \bibitem{NP24+}
 Wenbo Niu and Jinhyung Park,
\newblock \textit{Effective gonality theorem on weight-one syzygies of algebraic curves},
\newblock preprint (2024) arXiv:2405.13446.

\bibitem{OSCAR}
The OSCAR Team, 
\textit{OSCAR -- Open Source Computer Algebra Research system} Version 1.5.0 (2025), available at \url{https://www.oscar-system.org}.

\bibitem{Park24}
Euisung Park, \textit{On rank 3 quadratic equations of projective varieties}, Trans. Amer. Math. Soc. \textbf{377} (2024), 2049--2064.
	
\bibitem{Raicu} 
Claudiu Raicu,
\newblock \textit{$3\times 3$ Minors of Catalecticants},
\newblock Math. Res. Lett. \textbf{20} (2013), 745--756.
	
\bibitem{Ravi}
M.S. Ravi, \textit{Determinantal equations for secant varieties of curves}, Commun. Algebra \textbf{22} (1994), 3103--3106.

\bibitem{ScalaDiag}
Luca Scala, \textit{Notes on diagonals of the product and symmetric variety of a surface}, J. Pure Applied Algebra \textbf{224} (2020), 106352.

\bibitem{SeoYoon}
Haesong Seo and Chiwon Yoon, \textit{Secant variety and syzygies of Hilbert scheme of two points}, preprint (2024), arxiv:2403.07315.

\bibitem{Sheridan}
John Sheridan, \textit{Projective normality of canonical symmetric squares}, preprint (2022), arXiv:2211.08313.

\bibitem{SS2011}
Jessica Sidman and Gregory G. Smith, \textit{Linear determinantal equations for all projective schemes}, Algebra Number Theory \textbf{5} (2011), 1041--1061.

\end{thebibliography}

\end{document}